\documentclass[10pt]{amsart}

\usepackage{amsmath, amsthm}
\usepackage{eucal}
\usepackage{fullpage}

\usepackage{amssymb}
\usepackage{amscd}
\usepackage{palatino}
\usepackage{latexsym}
\usepackage{epsfig}
\usepackage{graphicx}
\usepackage{amsfonts}
\usepackage{psfrag}
\usepackage{color}
\usepackage{curves} 
\usepackage{times}
\usepackage{xcolor}

\usepackage{url}
\usepackage{cite}

\usepackage[margin=1in]{geometry}
\usepackage[pdftex]{hyperref}

\numberwithin{equation}{section}
\numberwithin{figure}{section}

\newtheorem{theorem}{Theorem}[section]
\newtheorem*{theorem*}{Theorem}
\newtheorem{lemma}[theorem]{Lemma}
\newtheorem{proposition}[theorem]{Proposition}

\newtheorem{remark}[theorem]{Remark}

\newtheorem{example}[theorem]{Example}

\theoremstyle{definition}
\newtheorem{definition}[theorem]{Definition}

\newcommand{\al}{\alpha}
\newcommand{\be}{\beta}
\newcommand{\la}{\lambda}

\newcommand{\lee}{\langle}
\newcommand{\ree}{\rangle}

\newcommand{\mf}{\mathbf}
\newcommand{\mr}{\mathrm}
\newcommand{\mc}{\mathcal}

\newcommand{\OO}{\mc{O}}

\newcommand{\C}{{\mathbb{C}}}

\newcommand{\Z}{{\mathbb{Z}}}

\newcommand{\R}{{\mathbb{R}}}
\newcommand{\N}{{\mathbb{N}}}

\renewcommand{\P}{{\mathbb{P}}}

\newcommand{\mfi}{\mf{i}}
\newcommand{\mfm}{\mf{m}}

\newcommand{\mfla}{\mf{\la}}

\DeclareMathOperator{\codim}{codim}

\begin{document}

\title{Newton-Okounkov bodies  of Bott-Samelson varieties and Grossberg-Karshon twisted cubes}

\author{Megumi Harada}
\address{Department of Mathematics and
Statistics\\ McMaster University\\ 1280 Main Street West\\ Hamilton, Ontario L8S4K1\\ Canada}
\email{Megumi.Harada@math.mcmaster.ca}
\urladdr{\url{http://www.math.mcmaster.ca/Megumi.Harada/}}

\author{Jihyeon Jessie Yang}
\address{Department of Mathematics and
Statistics\\ McMaster University\\ 1280 Main Street West\\ Hamilton, Ontario L8S4K1\\ Canada}
\email{jyang@math.mcmaster.ca}
\urladdr{\url{http://www.math.mcmaster.ca/~jyang/}}
\thanks{}

\keywords{Bott-Samelson variety, Newton-Okounkov bodies, path operators, generalized Demazure modules}
\subjclass[2000]{Primary:14M15; Secondary: 20G05}

\date{\today}


\begin{abstract}
We describe, under certain 
conditions, the Newton-Okounkov body
of a Bott-Samelson variety as a lattice polytope
defined by an explicit list of inequalities. The valuation that we use
to define the Newton-Okounkov body is different from that used
previously in the literature.  The polytope that arises is a special case of the 
Grossberg-Karshon twisted cubes studied by Grossberg and Karshon in connection to character formulae for irreducible $G$-representations and also studied previously by the authors 
in relation to certain toric varieties associated to Bott-Samelson varieties.
\end{abstract}

\maketitle

\setcounter{tocdepth}{1}

\tableofcontents

\section*{Introduction}

The main result of this paper is an explicit computation of a
Newton-Okounkov body associated to a Bott-Samelson variety, under
certain hypotheses. To place
our result in context, recall that the recent theory of
Newton-Okounkov bodies, introduced independently by Kaveh and
Khovanskii \cite{KavehKhovanskii2012} and Lazarsfeld and Mustata
\cite{LazarsfeldMustata2009}, associates to a complex algebraic
variety $X$ (equipped with some auxiliary data) a convex body of
dimension $n=\dim_\C(X)$. In some cases, this convex body (the
\textbf{Newton-Okounkov body}, also called \textbf{Okounkov body}) is a rational polytope; indeed, if $X$
is a projective toric variety, then one can recover the usual moment
polytope of $X$ as a Newton-Okounkov body. These Newton-Okounkov
bodies have been shown to be related to many other research areas,
including (but certainly not limited to) toric degenerations
\cite{Anderson2013}, representation theory \cite{Kaveh2011},
symplectic geometry \cite{HaradaKaveh2015}, and Schubert calculus
\cite{Kiritchenko2013, Kiritchenko2014}. However, relatively
few explicit examples of Newton-Okounkov bodies
have been computed so far, and thus it is an interesting problem
to give new and concrete examples. 

Motivated by the above, in this paper we study the Newton-Okounkov
bodies of Bott-Samelson varieties; these varieties are well-known and studied in representation
theory due to their relation to Schubert varieties and flag varieties
(see e.g. \cite{Demazure1974}) and have been studied in the
context of Newton-Okounkov bodies. For instance, Anderson computed a Newton-Okounkov
body for an $SL(3,\C)$ example in \cite{Anderson2013}, they appear
in the proof of Kaveh's identification of Newton-Okounkov bodies as string polytopes
in \cite{Kaveh2011}, and Kiritchenko conjectures a description of 
some Newton-Okounkov bodies of
Bott-Samelson varieties using her divided-difference operators in
\cite{Kiritchenko2013}. Moreover, the global Newton-Okounkov body of Bott-Samelson varieties is studied by Sepp\"anen and Schmitz in \cite{SeppanenSchmitz2014}, where they show that it is rational polyhedral and also give an inductive description of it. 
Additionally, during the preparation of this manuscript we learned that Fujita has
also (independently) computed the Newton-Okounkov bodies of Bott-Samelson varieties
\cite{Fujita2015}. However, 
the valuation which we use in this paper (part of the auxiliary data necessary for the definition of a Newton-Okounkov body) is
\emph{different} from that associated to the ``vertical flag'' considered by Sepp\"anen and Schmitz \cite{SeppanenSchmitz2014}, the highest-term valuation used by Fujita and Kaveh  \cite{Fujita2015, Kaveh2011} and the geometric valuation used by Anderson and Kiritchenko in
\cite{Anderson2013, Kiritchenko2013} (cf. also Remark~\ref{remark:valuations}).

We now briefly recall the geometric objects of interest; for details see
Section~\ref{sec:background}. Let $G$ be a complex semisimple connected and simply connected linear algebraic group
and let $\{\alpha_1, \ldots, \alpha_r\}$ denote the set of
simple roots of $G$. Let $\mfi=(i_1, \ldots, i_n) \in
\{1,2,\ldots,r\}^n$ be a \emph{word} which specifies a sequence of
simple roots $\{\alpha_{i_1}, \ldots, \alpha_{i_n}\}$. We say that a
word is reduced if the corresponding sequence of simple roots
gives a reduced word decomposition $s_{\alpha_{i_1}} s_{\alpha_{i_2}} \cdots s_{\alpha_{i_n}}$ of an element in the Weyl group. Also let
$\mfm=(m_1, \ldots, m_n) \in \Z^n_{\geq 0}$ be a \emph{multiplicity
  list}; this specifies a sequence of weights $\{\lambda_1 := m_1
\varpi_{\alpha_{i_1}}, \ldots, \lambda_n :=  m_n \varpi_{\alpha_{i_n}}\}$ in the
weight lattice of $G$. Associated to $\mfi$ and $\mfm$ one can define
a Bott-Samelson variety $Z_\mfi$
(cf. Definition~\ref{def:bott-samelson}) and a line bundle
$L_{\mfi,\mfm}$ over it (cf. Definition~\ref{definition:line
  bundle}). The spaces of global sections $H^0(Z_\mfi, L_{\mfi,\mfm})$
appear in representation theory as so-called \emph{generalized
  Demazure modules}. We also consider a certain natural flag of
subvarieties
$Y_{\bullet}: Z_\mfi = Y_0 \supseteq Y_1 \supseteq \cdots \supseteq Y_{n-1}
\supseteq Y_n = \{\textup{pt}\}$ in $Z_\mfi$ and consider a valuation
$\nu_{Y_\bullet}$ on the spaces of sections $H^0(Z_\mfi,
L_{\mfi,\mfm}^{\otimes k})$ associated to $Y_\bullet$ (for details see
Section~\ref{section:NOBY}). Our main result
is the following; a more precise statement is given in
Theorem~\ref{theorem:main}. The polytope $P(\mfi,\mfm)$ and the ``condition \textbf{(P)}'' mentioned
in the statement of the theorem are discussed below. 

\begin{theorem*}
Let $\mfi=(i_1,\ldots,i_n) \in \{1,2,\ldots,r\}^n$ be a word and $\mfm=(m_1,\ldots,m_n) \in \Z^n_{\geq 0}$ be a
multiplicity list. Let $Z_\mfi$ and $L_{\mfi,\mfm}$ be the associated
Bott-Samelson variety and line bundle. Suppose that $\mfi$ is reduced
and the pair $(\mfi,\mfm)$ satisfies
condition \textbf{(P)}. Then 
the Newton-Okounkov body $\Delta(Z_\mfi, L_{\mfi,\mfm}, \nu_{Y_\bullet})$ of $Z_\mfi$, with respect
to the line bundle $L_{\mfi,\mfm}$ and the geometric valuation
$\nu_{Y_\bullet}$, is equal to $P(\mfi,\mfm)$ (up
to a reordering of coordinates). 
\end{theorem*}

Both the polytope $P(\mfi,\mfm)$ and the ``condition \textbf{(P)}'' (defined precisely in Section~\ref{sec:bijection}) mentioned in the theorem 
have appeared previously in the literature. Indeed, 
the polytope $P(\mfi,\mfm)$ is a special case of the \emph{Grossberg-Karshon
  twisted cubes} which yield character formulae (possibly with sign) for
irreducible $G$-representations \cite{GrossbergKarshon1994}. Specifically, we showed
in \cite[Proposition 2.1]{HaradaYang2014} that if the pair $(\mfi,\mfm)$ satisfies
condition \textbf{(P)}, then the Grossberg-Karshon twisted cube is
equal to the polytope $P(\mfi,\mfm)$ and the Grossberg-Karshon
character formula from \cite{GrossbergKarshon1994} corresponding to
$\mfi$ and $\mfm$ is a \emph{positive} formula (i.e. with no negative
signs). We also related the condition \textbf{(P)} to the geometric
condition that a certain torus-invariant divisor $D$ in a toric variety
related to $Z_\mfi$ is basepoint-free \cite[Theorem
2.4]{HaradaYang2014}. For the purposes of the present manuscript, it is also significant that the polytope $P(\mfi, \mfm)$ is a \emph{lattice} polytope (not just a rational polytope) whose vertices can be easily described as the Cartier data of the torus-invariant divisor $D$ mentioned above \cite[Theorem 2.4]{HaradaYang2014}. Thus our theorem gives a computationally efficient description of the Newton-Okounkov body $\Delta(Z_\mfi, L_{\mfi,\mfm}, \nu_{Y_\bullet})$. 

We now sketch the main ideas in the proof of our main result
(Theorem~\ref{theorem:main}). To place the discussion in context, it may be useful to
recall that an essential step in the computation of a Newton-Okounkov
body of a variety $X$ is to compute a certain semigroup $S = S(R,\nu)$ associated to
the (graded) ring of sections $R=\oplus_k H^0(X, L^{\otimes k})$ for 
$L$ is a line bundle over $X$ and a choice of valuation $\nu$. In
general, this computation can be quite subtle; one of the main
difficulties is that the semigroup may not even be finitely
generated. (The issue of finite generation, in the context of
Newton-Okounkov bodies, is studied in \cite{Anderson2013}.) Even when
$S$ is finitely generated, finding explicit generators is related to
the problem of finding a ``SAGBI basis'' for $R$ with respect to the
valuation\footnote{Such a basis is also called a \textit{Khovanskii basis} in \cite[Section
8]{HaradaKaveh2015}, cf. also \cite[Section
5.6]{KavehKhovanskii2008}.} which appears to be non-trivial in
practice. In this manuscript, we are able to sidestep this subtle
issue and compute $S$ directly by a simple observation which we now
explain. It is a general fact that the valuations arising from flags
of subvarieties $Y_\bullet$ such as those above have
\emph{one-dimensional leaves}
(cf. Definition~\ref{definition:valuation}). It is also an elementary
fact that a valuation $\nu$ with one-dimensional leaves, defined on a
finite-dimensional vector space $V$, satisfies $\lvert \nu(V \setminus
\{0\}) \rvert = \dim_\C(V)$ \cite[Proposition
2.6]{KavehKhovanskii2012}. As it happens, in our setting the vector
spaces in question are precisely the generalized Demazure modules
$H^0(Z_\mfi, L_{\mfi,\mfm})$ mentioned above, and
Lakshmibai-Littelmann-Magyar prove in \cite{LakLitMag2002} that
$\dim_\C(H^0(Z_\mfi, L_{\mfi,\mfm})) = \lvert \mc{T}(\mfi,\mfm)
\rvert$ where $\mc{T}(\mfi,\mfm)$ is the set of \textbf{standard
  tableaux} associated to $\mfi$ and $\mfm$. Armed with this key
theorem of Lakshmibai-Littelmann-Magyar, we are able to compute our
semigroup $S$ and hence the Newton-Okounkov body explicitly in two
steps. On the one hand, we show in
Proposition~\ref{proposition:subset} that, assuming $\mfi$ is reduced, 
our geometric valuation
$\nu_{Y_\bullet}$ defined on $H^0(Z_\mfi, L_{\mfi,\mfm}) \setminus
\{0\}$ takes values in the polytope $P(\mfi,\mfm)$ (up to reordering
coordinates). On the other hand, we show in
Proposition~\ref{proposition:bijection} that, assuming that
$(\mfi,\mfm)$ satisfies condition \textbf{(P)}, there is a bijection
between the lattice points in $P(\mfi,\mfm)$ and the set of standard
tableaux $\mc{T}(\mfi,\mfm)$, so in particular $\lvert P(\mfi,\mfm)
\cap \Z^n \rvert = \lvert \mc{T}(\mfi, \mfm) \rvert$. Now a simple
counting argument and the fact that $P(\mfi,\mfm)$ is a lattice polytope finishes the proof of the main theorem.

We now outline the contents of the manuscript. In
Section~\ref{sec:background} we establish basic
terminology and notation, and also state the key result of
Lakshmibai-Littelmann-Magyar (Theorem~\ref{theorem:LLM}). The
statement and proof of the bijection between $\mc{T}(\mfi,\mfm)$ and
the lattice points in $P(\mfi,\mfm)$ occupies
Section~\ref{sec:bijection}. In the process we introduce a separate
``condition \textbf{(P')}'', stated directly in the language of paths and
root operators as in \cite{Littelmann1994, Littelmann1995, LakLitMag2002}, and prove in
Proposition~\ref{prop:P implies P'} that our polytope-theoretic
condition \textbf{(P)} implies condition \textbf{(P')}. It is then
straightforward to see that condition \textbf{(P')} implies that 
$\lvert P(\mfi,\mfm) \cap \Z^n \rvert =\lvert \mc{T}(\mfi,\mfm) \rvert$.
In Section~\ref{section:NOBY} we recall
in some detail the definition of a Newton-Okounkov body and define our
geometric valuation $\nu_{Y_\bullet}$ with respect to a certain flag
of subvarieties. We then prove in
Proposition~\ref{proposition:subset} that
$\nu_{Y_\bullet}$ takes values in our
polytope; as already explained, by using the bijection from Section~\ref{sec:bijection} our
main theorem then readily follows. Concrete examples and pictures for
$G=SL(3,\C)$ are contained in Section~\ref{sec:examples}.

We take a moment to comment on the combinatorics in
Section~\ref{sec:bijection}. It may well be that our polytope
$P(\mfi,\mfm)$, our conditions \textbf{(P)} and \textbf{(P')}, and our
Proposition~\ref{proposition:bijection}, are well-known 
or are minor variations on standard arguments in
combinatorial representation theory. However, we were unable to locate
exact references. 
We welcome comments from the experts. At any rate, as the discussion above indicates, Proposition~\ref{proposition:bijection}
is only a stepping stone to our main result 
(Theorem~\ref{theorem:main}). One final comment: in
Section~\ref{sec:bijection}, we chose to explain conditions
\textbf{(P)} and \textbf{(P')} separately, and to explicitly state and
prove the relation between them in Proposition~\ref{prop:P implies P'},
because we suspect that condition \textbf{(P')} may be more familiar to
experts in representation theory, whereas our condition \textbf{(P)} arises
from the toric-geometric considerations in
\cite{HaradaYang2014}. Put another way, our condition
\textbf{(P)} is a \emph{geometrically} motivated condition on $\mfi$ and
$\mfm$ which suffices to guarantee the condition \textbf{(P')}.

Finally, we mention some directions for future work. 
Firstly, we hope to better understand the relation between our computations and those in \cite{Fujita2015}. Secondly, our condition \textbf{(P)} on the pairs $(\mfi,\mfm)$
is rather restrictive and the 
corresponding Newton-Okounkov bodies are combinatorially
extremely simple (they are essentially cubes, though they can sometimes
degenerate). Hence it is a natural problem to ask for the relation, if
any, between the Newton-Okounkov bodies computed in this paper and
those for the line bundles which do \emph{not} satisfy the condition
\textbf{(P)}. It may be possible to analyze such a relationship using some results of
Anderson \cite{Anderson2013}, and we hope to take this up
in a future paper. Thirdly, it would be of
interest to examine the relation between our polytopes $P(\mfi,\mfm)$
and the polytopes arising from Kiritchenko's divided-difference
operators, particularly in relation to her ``degeneration of string
spaces'' in \cite[Section 4]{Kiritchenko2013}.

\medskip
\noindent \textbf{Acknowledgements.} 
We are grateful to Lauren DeDieu, Naoki Fujita, Dmitry Kerner,  Eunjeong Lee, and
Satoshi Naito for useful conversations, and to Henrik Sepp\"anen for explaining his work to us. We especially
thanks Dave Anderson for help with 
our arguments in Section~\ref{section:NOBY}, and  
to Kiumars Kaveh for pointing out a critical error in a
previous version of this manuscript. 
The first author was partially supported by an NSERC Discovery Grant,
a Canada Research Chair (Tier 2) Award, an Ontario Ministry of Research
and Innovation Early Researcher Award, an Association for Women in Mathematics Ruth Michler Award, 
and a Japan Society for the Promotion of Science Invitation Fellowship for
Research in Japan (Fellowship ID L-13517). 
Both authors thank the Osaka City University
Advanced Mathematics Institute for its hospitality while part of this
research was conducted. The first author also thanks the Department of
Mathematics at Cornell University for its hospitality during her
tenure as the Ruth Michler Visiting Fellow, when portions of this
manuscript were written. The second author also thanks Dong Youp Suh for his hospitality during her visit to the Department of Mathematical Sciences, KAIST.

\section{Preliminaries}\label{sec:background}

In this section we record basic notation in
Section~\ref{sec:notation}, recall the definitions of the central
geometric objects in Section~\ref{sec:bott}, and state a key result (Theorem~\ref{theorem:LLM})
of Lakshmibai, Littelmann, and Magyar in Section~\ref{sec:paths}.

\subsection{Notation}\label{sec:notation}

We list here some notation and conventions to be used in the manuscript. 

\begin{itemize}

\item $G$ is a complex semisimple connected and simply connected algebraic group over $\C$ and $\mathfrak{g}$ denotes its Lie algebra.

\item $H$ is a Cartan subgroup of $G$.

\item $B$ is a Borel subgroup of $G$ with $H \subset B \subset G$.  

\item $r$ is the rank of $G$.

\item $X$ denotes the weight lattice of $G$ 
and $X_\R = X \otimes_\Z \R$ is its real form. The Killing form\footnote{The Killing form is naturally defined on the Lie algebra of $G$ but its restriction to the Lie algebra $\mathfrak{h}$ of $H$ is positive-definition, so we may identify $\mathfrak{h} \cong \mathfrak{h}^*$.} 
on $X_\R$ is denoted by
 $\langle \alpha, \beta \rangle$.

\item For a weight $\alpha \in X$, we let $e^\alpha$ denote the corresponding multiplicative character $e^\alpha: H \to \C^*$. 

\item $\{\alpha_1, \ldots, \alpha_r\}$ is the set of positive simple roots (with an ordering) with respect to the choices $H \subset B \subset G$ and $\{\alpha^\vee_1, \ldots, \alpha^\vee_r\}$ are 
the corresponding coroots. Recall that the coroots satisfy 
\[
\alpha^\vee := \frac{2 \alpha}{\langle \alpha,\alpha \rangle}.
\]
In particular, $\langle \alpha,
\alpha^\vee \rangle = 2$ for any simple root $\alpha$. 

\item For a simple root $\alpha$ 
let $s_\alpha:X\rightarrow X,
\lambda\mapsto\lambda-\langle\lambda,\al^{\vee}\rangle \al,$ be the associated simple reflection; these generate the Weyl group $W$.

\item $\{\varpi_1, \ldots, \varpi_r\}$ is the set of fundamental weights satisfying $\langle \varpi_i, \alpha^\vee_j \rangle = \delta_{i,j}$.

\item For
 a simple root $\alpha$, $P_\alpha := B \cup
  B s_\alpha B$ is the
  minimal parabolic subgroup containing $B$ associated to $\alpha$.

\end{itemize}

\subsection{Bott-Samelson varieties}\label{sec:bott}

In this section, we briefly recall the definition of Bott-Samelson varieties and 
some facts about line bundles on Bott-Samelson varieties. Further
details may be found, for instance, in \cite{GrossbergKarshon1994}. Note that the
literature uses many different notational conventions.

With the notation in Section~\ref{sec:notation} in place, suppose 
given an arbitrary {\em word} in $\{1,2,\ldots, r\}$, i.e. a sequence 
$\mfi=(i_1,\dots,i_n)$ with $1 \leq i_j \leq r$. This specifies an
associated sequence of 
simple roots $\{\alpha_{i_1},
\alpha_{i_2}, \ldots, \alpha_{i_n}\}$. To simplify notation we define
$\beta_j := \alpha_{i_j}$, so the sequence above can be denoted
$\{\beta_1, \ldots, \beta_n\}$. Note that we do not assume here that the corresponding expression
$s_{\beta_1} s_{\beta_2} \cdots s_{\beta_n}$ is 
reduced; in particular, there may be repetitions. (However, we will add the
reducedness as a hypothesis in Section~\ref{section:NOBY}.) 

\begin{definition}\label{def:bott-samelson}
The \textbf{Bott-Samelson variety} corresponding to a word $\mfi=(i_1, \dots, i_n) \in \{1,2,\ldots,r\}^n$ is the quotient
\[Z_{\mfi}:=(P_{\beta_1}\times\cdots\times
P_{\beta_n})/B^n\]
where $\be_j = \alpha_{i_j}$ and $B^n$ acts on the right on $P_{\beta_1} \times \cdots \times
P_{\beta_n}$ by:
\[(p_1,\dots,p_n)\cdot
(b_1,\dots,b_n):=(p_1b_1,b_1^{-1}p_2b_2,\dots,b_{n-1}^{-1}p_nb_n).\]
\end{definition}

It is known that $Z_{\mfi}$ is a smooth projective algebraic variety of dimension
$n$. 
By convention, if $n=0$ and $\mfi$ is the empty word, we set
$Z_{\mfi}$ equal to a point.

We next describe certain line bundles over a Bott-Samelson variety. 
Suppose given a sequence $\{\mfla_1,\dots,\mfla_n\}$ of weights
$\mfla_j \in X$. We let $\C^*_{(\mfla_1,\dots,\mfla_n)}$ denote the one-dimensional representation of $B^n$ defined by 
\begin{equation}\label{def:rep}
(b_1,\dots,b_n)^{-1}\cdot k:=e^{\mfla_1}(b_1)\cdots e^{\mfla_n}(b_n)k.
\end{equation}
(Notice that this is isomorphic to the representation
$\C_{(-\mfla_1,\cdots,-\mfla_n)}$.)

\begin{definition}\label{definition:line bundle}
Let $\lambda_1, \ldots, \lambda_n$ be a sequence of weights. We define
the line bundle $L_{\mfi}(\lambda_1, \ldots, \lambda_n)$ over
$Z_{\mfi}$ to be 
\begin{equation}\label{eq:definition line bundle} 
L_{\mfi}(\lambda_1, \ldots, \lambda_n):=(P_{\be_1}\times\cdots
P_{\be_n})\times_{B^n} \C^*_{(\mfla_1,\dots,\mfla_n)}
\end{equation}
where the equivalence relation is given by 
 \[((p_1, \ldots, p_n) \cdot (b_1, \ldots, b_n),k)\sim ((p_1, \ldots,
 p_n), (b_1, \ldots, b_n) \cdot k)\]
 for $(p_1, \ldots, p_n) \in P_{\beta_1} \times\cdots\times
 P_{\beta_n}, (b_1, \ldots, b_n) \in
 B^n$, and $k\in\C$. The projection $L_{\mfi}(\lambda_1, \ldots, \lambda_n) \to Z_{\mfi}$
 to the base space is given by taking the first factor
 $[(p_1, \ldots, p_n, k)] \mapsto [(p_1, \ldots, p_n)] \in Z_{\mfi}$. 
\end{definition}
In what follows, we will frequently choose the weights $\lambda_j$ to be of a special form. 
Specifically, suppose given a \textbf{multiplicity list} 
$\mfm=(m_1,\dots,m_n) \in \Z_{\geq 0}^n$. Then we may define a
sequence of weights $\{\mfla_1, \ldots, \mfla_n\}$ associated to the
word $\mfi$ and the multiplicity list $\mfm$ by setting 
\begin{equation}\label{def:lambda}
\mfla_1:=m_1\varpi_{i_1},\dots, \mfla_n=m_n\varpi_{i_n}.
\end{equation}
In this special case we will use the notation 
\begin{equation}\label{eq:definition Lim}
L_{\mfi,\mfm} := L_{\mfi}(m_1 \varpi_{\beta_{i_1}},
\cdots, m_n \varpi_{\beta_{i_n}}).
\end{equation}

In this manuscript we will study the space of global sections of these line bundles. Note that 
the Borel subgroup acts on both $Z_{\mfi}$ and $L_{\mfi,\mfm}$ by left multiplication on the first coordinate: indeed, 
for $b\in B$, the equation $b\cdot [(p_1,\dots,p_n)]:=[(bp_1,p_2,\dots, p_n)]$ defines the action on $Z_{\mfi}$ and 
$b\cdot [(p_1,\dots,p_n,k)]:=[(bp_1,p_2,\dots, p_n,k)]$ defines the
action on $L_{\mfi,\mfm}$. It is straightforward to check that both
are well-defined. 
The space of global sections $H^0(Z_{\mfi},L_{\mfi,\mfm})$ is then
naturally a $B$-module; these are called \textit{generalized Demazure
  modules} (cf. for instance \cite{LakLitMag2002}).

\subsection{Paths and root operators}\label{sec:paths}

We use the machinery of paths and root operators as in 
\cite{LakLitMag2002} (cf. also \cite{Littelmann1994, Littelmann1995}) 
so in this section we briefly recall some necessary definitions and
basic properties. 

Let $X_{\R}:=X\otimes_{\Z}\R$ denote the real form of the
weight lattice.  
By a \textbf{path} we will mean a piecewise-linear map $\pi:[0,1]\rightarrow X_{\R}$ (up to reparametrization) with $\pi(0)=0$. 
We consider the set $\Pi \cup \{\mathbf{O}\}$ where
$\Pi$ denotes the set of all paths and $\mathbf{O}$ is a formal symbol.
For a weight $\la\in X$, we let $\pi^{\la}$ denote the straight-line path: $\pi^{\la}(t):= t\la$. 
By the symbol $\pi_1 \star \pi_2$ we mean the concatenation of two paths; more precisely, $\pi(t) = (\pi_1  \star \pi_2)(t)$ is defined by 
\begin{equation}\label{eq:concatenation}
\pi(t) := 
\begin{cases} 
\pi_1(2t) \textup{ if } 0 \leq t \leq 1/2 \\
\pi_1(1) + \pi_2(2t-1) \textup{ if } 1/2 \leq t \leq 1.
\end{cases} 
\end{equation} 
By convention we take $\pi \star \mathbf{O}
:= \pi$ for any element $\pi \in \Pi \cup \{\mathbf{O}\}$. 
For a simple root $\alpha$ and a path $\pi$, we define $s_\alpha(\pi)$
to be the path
given by $s_\alpha(\pi)(t) := s_\alpha(\pi(t))$, i.e., the path
$\pi$ is reflected by $s_\alpha$. 
We pay particular
attention to endpoints so we give it a name: given $\pi$ we say
the \textbf{weight} of $\pi$ is its endpoint, $\mr{wt}(\pi):=\pi(1)$
(also denoted $v(\pi)$ in the literature, see \cite{Littelmann1994}). 
The following is immediate from the definitions. 
\begin{lemma}\label{lemma:weights}
Let $\pi, \pi_1, \pi_2$ be paths in $\Pi$ and $\alpha$ a simple
root. Then $\mr{wt}(\pi_1 \star \pi_2) = \mr{wt}(\pi_1)+\mr{wt}(\pi_2)$ and
$\mr{wt}(s_\alpha(\pi) = s_\alpha(\mr{wt}(\pi))$.
\end{lemma}

Fix a simple root
$\alpha$. We now briefly recall the definitions of the \textbf{raising operator $e_\alpha$} and
\textbf{lowering operator $f_\alpha$} on the set $\Pi \cup
\{\mathbf{O}\}$, for which we need some preparation of notation. Fix a path $\pi \in \Pi$. We cut $\pi$ into 3 pieces
according to the behavior of the path $\pi$ under the projection with
respect to $\alpha$. More precisely, define the function 
\[
h_\alpha: [0,1] \to \R,   t \mapsto \langle \pi(t), \alpha^{\vee}
\rangle
\]
and let $Q$ denote the smallest integer attained by $h_\alpha$, i.e., 
\[
Q := \min \{ \textup{image}(h_\alpha) \cap \Z \}. 
\]
Note that since $\pi(0)=0$ by definition we always have $Q \leq
0$. Now let $q := \min \{ t \in [0,1]: h_\alpha(t) = \langle \pi(t),
\alpha^{\vee} \rangle = Q \}$ be the ``first'' time $t$ at which the
minimum integer value of $h_\alpha$ is attained. Next, in the case that $Q \leq
-1$ (note that if $Q = 0$ then, since $\pi(0)=0$, the value $q$ must
be $0$ and the following discussion is not applicable) then we define
$y$ to be the ``last time before $q$'' when the value $Q+1$ is
attained. More precisely, $y$ is defined by the conditions 
\[
h_\alpha(y) = Q+1, \textup{ and } Q < h_\alpha(t) < Q+1 \textup{ for }
y < t < q.
\]
We now define three paths $\pi_1, \pi_2, \pi_3$ in such a way that
$\pi$ is by definition the concatenation $\pi = \pi_1 \star \pi_2
\star \pi_3$, where $\pi_1$ is the path $\pi$ ``up to time $y$'',
$\pi_2$ is the path $\pi$ ``between $y$ and $q$'', and $\pi_3$ is the
path $\pi$ ``after time $q$''. More precisely, we define 
\[
\pi_1(t) := \pi(ty) \textup{ and } \pi_2(t) := \pi(y+t(q-y)) - \pi(y),
\textup{ and } 
\pi_3(t) := \pi(q+t(1-q)) - \pi(q).
\]
See \cite[Example, Section 1.2]{Littelmann1994} for a figure
illustrating an example in rank 2. 
Given this decomposition of $\pi$ into ``pieces'', we may now define
the \textbf{raising (root) operator $e_\alpha$} as follows. 

\begin{definition}\label{definition:raising operator}
Fix a path $\pi$.
If $Q=0$, i.e. if the path $\pi$ lies entirely in the closed
half-space defined by $\{h_\alpha > -1\}$, then $e_\alpha(\pi)=\mathbf{O}$,
where here $\mathbf{O}$ is the formal symbol in $\Pi \cup
\{\mathbf{O}\}$. 
If $Q<0$, then we define $e_\alpha(\pi) := \pi_1 \star
s_\alpha(\pi_2) \star \pi_3$, i.e. we ``reflect across $\alpha$'' the
portion of the path $\pi$ between time $y$ and time $q$. 
We also define $e_\alpha(\mathbf{O}) = \mathbf{O}$. 
\end{definition}

The \textbf{lowering (root) operator $f_\alpha$} 
may be defined similarly. This time, let $p$ denote the \emph{maximal}
real number in $[0,1]$ such that $h_\alpha(p)=Q$, i.e., it is the
``last'' time $t$ at which the minimal value $Q$ is attained. Then let
$P$ denote the integral part of $h_\alpha(1) - Q$. If $P \geq 1$, then
let $x$ denote the first time after $p$ that $h_\alpha$ achieves the
value $Q+1$; more precisely, let $x$ be the unique element in $(p,1]$
satisfying 
\[
h_\alpha(x)=Q+1 \textup{ and } Q < h_\alpha(t) < Q+1 \textup{ for } p
< t < x.
\]
Once again we may decompose the path $\pi$ into 3 components, $\pi =
\pi_1 \star \pi_2 \star \pi_3$ by the equations 
\begin{equation}\label{eq:definitions of pi_1 pi_2 pi_3}
\pi_1(t) := \pi(tp) \textup{ and } \pi_2(t) := \pi(p+t(x-p)) - \pi(p)\textup{
  and } \pi_3(t) := \pi(x+t(1-x)) - \pi(x).
\end{equation}
Given this decomposition, we define the \textbf{lowering (root)
  operator $f_\alpha$} as follows. 

\begin{definition}\label{definition:f_beta}
Fix a path $\pi$
as above. 
If $P \geq 1$, then we define $f_\alpha(\pi) := \pi_1 \star
s_\alpha(\pi_2) \star \pi_3$, so we 
``reflect across $\alpha$'' the
portion of the path $\pi$ between time $p$ and $x$. 
If $P=0$, 
then 
$f_\alpha(\pi)=\mathbf{O}$. Finally, we define $f_\alpha(\mathbf{O}) =
\mathbf{O}$. 
\end{definition}

The following basic properties of the root operators are recorded in
\cite[Section 1.4]{Littelmann1994}. 

\begin{lemma}\label{lem:Littelmann}
Let $\pi \in \Pi$ be a path. 
  \begin{enumerate}
    \item If $e_{\al}(\pi)\ne \mathbf{O}$, then $\mr{wt}(e_{\al}(\pi))=\mr{wt}(\pi)+\al$, and if $f_{\al}(\pi)\ne \mathbf{O}$, then $\mr{wt}(f_{\al}(\pi))=\mr{wt}(\pi)-\al$.
    \item If $e_{\al}(\pi)\ne \mathbf{O}$, then $f_{\al}(e_{\al}(\pi))=\pi.$ If $f_{\al}(\pi)\ne \mathbf{O}$, then $e_{\al}(f_{\al}(\pi))=\pi$.
    \item $e_{\al}^n(\pi)=\mathbf{O}$ if and only if $n>-Q$, and
      $f_{\al}^n\pi=\mathbf{O}$ if and only if $n>P$. 
 \end{enumerate}
\end{lemma}

We now recall a result (Theorem~\ref{theorem:LLM} below) of Lakshmibai, Littelmann, and Magyar
\cite{LakLitMag2002} which is crucial to our arguments in the
remainder of this paper. 
Specifically, Theorem~\ref{theorem:LLM} gives a bijective correspondence between a
certain set $\mc{T}(\mfi, \mfm)$ of standard tableaux, defined below
using paths and the root operators, and a 
basis of the vector space $H^0(Z_\mfi, L_{\mfi,\mfm})$ of global sections
of $L_{\mfi,\mfm}$ over $Z_\mfi$. 
Our main result in Section~\ref{sec:bijection} is that -- under certain conditions
on the word $\mfi$ and the multiplicity list $\mfm$ -- 
there exists, in turn, a bijection between $\mc{T}(\mfi,\mfm))$ and the set
of integer lattice
points in a certain polytope. This then allows us to
compute Newton-Okounkov bodies associated to $Z_\mfi$ and $L_{\mfi,
  \mfm}$ in Section~\ref{section:NOBY}. 

We now recall the definition of standard tableaux. 
Suppose given a word $\mfi$ and multiplicity list $\mfm$ as above. Let
$\{\beta_1=\alpha_{i_1}, \ldots, \beta_n=\alpha_{i_n}\}$ be the sequence of simple roots
associated to $\mfi$ and set $\lambda_j := m_j \be_j$ for $1 \leq j \leq n$. The following 
is from \cite[Section 1.2]{LakLitMag2002}.

\begin{definition}\label{definition:standard tableau}
A path $\pi \in \Pi$ is called a
\textbf{(constructable) standard tableau of shape $\lambda =
  (\lambda_1, \ldots, \lambda_n)$} if there exist integers $\ell_1,
\ldots, \ell_n \in \Z_{\geq 0}$ such that 
\[
\pi = f_{\beta_1}^{\ell_1}(\pi^{\lambda_1} \star
f_{\beta_2}^{\ell_2}(\pi^{\lambda_2} \star \cdots
f_{\beta_n}^{\ell_n}(\pi^{\lambda_n}) \cdots )
\]
where the $f_{\beta_j}$ are the lowering operators defined
above. Given a word $\mfi=(i_1, \ldots, i_n)$ and multiplicity list
$\mfm=(m_1,\ldots,m_n)$, we denote by $\mc{T}(\mfi,\mfm)$ the set of
standard tableau of shape $(\lambda_1=m_1 \varpi_{\beta_1}, \ldots,
\lambda_n=m_n \varpi_{\be_n})$. 
\end{definition}

It turns out that there are only finitely many standard tableau of a given shape $\lambda$
associated to a given pair $(\mfi, \mfm)$.
In fact, 
Lakshmibai, Littelmann, and Magyar prove
\cite[Theorems 4 and 6]{LakLitMag2002} the following.

\begin{theorem}\label{theorem:LLM}
Let $\mfi = (i_1, \ldots, i_n) \in \{1,\ldots, r\}^n$ be a word and
$\mfm = (m_1, \ldots, m_n) \in \Z^n_{\geq 0}$ be a multiplicity
list. Let
$\{\beta_1=\alpha_{i_1}, \ldots, \beta_n=\alpha_{i_n}\}$ be the sequence of simple roots
associated to $\mfi$ and set $\lambda_j := m_j \be_j$ for $1 \leq j \leq n$.
Then 
\[\lvert \mc{T}(\mfi, \mfm) \rvert = \dim_\C H^0(Z_\mfi, L_{\mfi,\mfm}).\]
\end{theorem}

\section{A bijection between standard tableaux and lattice points in a
polytope}\label{sec:bijection}

The main result of this section
(Proposition~\ref{proposition:bijection}) is that, under a certain 
assumption on the word $\mfi$ and the multiplicity list $\mfm$, there is a bijection between the set
of integer lattice points within a certain lattice polytope $P(\mfi,
\mfm)$ and the set of standard tableaux $\mc{T}(\mfi,\mfm)$. Together
with Theorem~\ref{theorem:LLM} this then implies that the cardinality
of $P(\mfi, \mfm) \cap \Z^n$ is equal to the dimension of the space
$H^0(Z_{\mfi}, L_{\mfi,\mfm})$ of sections of the line
bundle $L_{\mfi, \mfm}$ over the Bott-Samelson
variety $Z_{\mfi}$. 
This then allows us to compute Newton-Okounkov bodies in the next
section. The necessary hypothesis on $\mfi$ and $\mfm$, which we call
\textbf{``condition (P)''}, also appeared in our previous work
\cite{HaradaYang2014} connecting the polytopes $P(\mfi,\mfm)$ with
representation theory and toric geometry
(cf. Remark~\ref{remark:toric} below).

We begin with the definition of the polytope $P(\mfi,\mfm)$ by 
an explicit set of inequalities. 

\begin{definition}\label{def:polytope}
Let $\mfi = (i_1, \ldots, i_n) \in \{1,\ldots, r\}^n$ be a word and
$\mfm = (m_1, \ldots, m_n) \in \Z^n_{\geq 0}$ be a multiplicity
list. Then the polytope $P(\mfi,\mfm)$ is defined to be the 
set of all real points $(x_1,\dots,x_n) \in \R^n$ satisfying  the following inequalities: 
\[\begin{array}{ccccl}
0&\le& x_n &\le  &A_n:= m_n,\\
0&\le &x_{n-1}&\le & A_{n-1}(x_n):= \lee m_{n-1}\varpi_{\be_{n-1}}+m_n\varpi_{\be_n}-x_n\be_n,\be_{n-1}^{\vee}\ree,\\
0&\le& x_{n-2}&\le & A_{n-2}(x_{n-1},x_n):=\lee m_{n-2}\varpi_{\be_{n-2}}+m_{n-1}\varpi_{\be_{n-1}}+m_{n}\varpi_{\be_n}-x_{n-1}\be_{n-1}-x_n\be_n,\be_{n-2}^{\vee}\ree\\
  &   &\vdots& &\\
0&\le&x_1&\le& A_1(x_2,\dots,x_n):=\lee m_1\varpi_{\be_1}+m_2\varpi_{\be_2}+\cdots + m_n\varpi_{\be_n}-x_2\be_2-\cdots-x_n\be_n,\be_1^{\vee}\ree
\end{array}\]
\end{definition}

\begin{remark}\label{remark:toric} 
\begin{itemize} 
\item The polytopes $P(\mfi,\mfm)$ have appeared previously in the
  literature and has connections to toric geometry and representation theory. Specifically, under a hypothesis on $\mfi$ and $\mfm$
  which we call ``condition \textbf{(P)}'' (see Definition~\ref{definition:conditionP}), we show
  in \cite{HaradaYang2014} that $P(\mfi,\mfm)$ is exactly a so-called
  \textit{Grossberg-Karshon twisted cube}. These twisted cubes were introduced in
  \cite{GrossbergKarshon1994} in connection with Bott towers and
  character formulae for irreducible $G$-representations. Our proof of
  this fact in \cite{HaradaYang2014} used a certain torus-invariant
  divisor in a toric variety associated to Bott-Samelson varieties
  studied by Pasquier \cite{Pasquier2010}.
\item The functions $A_k(x_{k+1},\ldots,x_n)$ appearing in
Definition~\ref{def:polytope} also have a natural interpretation in terms
of paths, as we shall see in Lemma~\ref{lemma:endpoints} below; this is useful in our proof of
Proposition~\ref{proposition:bijection}. 
\end{itemize} 
\end{remark}

In the statement of our main proposition of this section, we need the
following technical hypothesis on the word and the multiplicity
list. As noted above, the same condition appeared in our previous work
\cite{HaradaYang2014} which related the polytope $P(\mfi,\mfm)$ to 
toric geometry and representation theory.

\begin{definition}\label{definition:conditionP}
Let $\mfi = (i_1, \ldots, i_n) \in \{1,\ldots, r\}^n$ be a word and
$\mfm = (m_1, \ldots, m_n) \in \Z^n_{\geq 0}$ be a multiplicity
list.
We say that the pair $(\mfi, \mfm)$ 
  \textbf{satisfies condition (P)} if
  \begin{enumerate}
  \item[(P-n)] $m_n \geq 0$
  \end{enumerate}
  and for every integer $k$ with $1 \leq k \leq n-1$, the following statement, which we refer to as condition (P-k),
    holds:
\begin{enumerate}
\item[(P-k)]
      if $(x_{k+1}, \ldots, x_n)$ satisfies
     \begin{equation*}
      \begin{array}{ccccl}
      0& \leq& x_n& \leq& A_n \\
      0& \leq &x_{n-1}& \leq& A_{n-1}(x_n) \\
    && \vdots&& \\
     0 &\leq& x_{k+1}& \leq& A_{k+1}(x_{k+2}, \ldots, x_n), \\
       \end{array}
       \end{equation*}
   then \[A_k(x_{k+1},\ldots, x_n) \geq 0.\] \end{enumerate}
In particular, condition \textbf{(P)} holds if and only if the conditions (P-1) through (P-n) all hold.
\end{definition}

\begin{remark}
  The condition \textbf{(P)} is rather restrictive. On the other hand,
  for a given word $\mfi$, it is not
  difficult to explicitly construct (either directly from the definition, or by using the
  other 
  equivalent characterizations of condition \textbf{(P)} in
  \cite[Proposition 2.1]{HaradaYang2014})
  many choices of
  $\mfm$ such that $(\mfi,\mfm)$ satisfies condition
  \textbf{(P)}. 
\end{remark}

We may now state the main result of this section. 

\begin{proposition}\label{proposition:bijection}
 If $(\mfi, \mfm)$ satisfies condition \textbf{(P)}, then there
 exists a bijection between the set of integer lattice points in the
 polytope $P(\mfi,\mfm)$ and the set of standard tableaux
 $\mc{T}(\mfi,\mfm)$, therefore,
\[
\lvert P(\mfi,\mfm) \cap \Z^n \rvert = \lvert \mc{T}(\mfi, \mfm) \rvert.
\]
\end{proposition}

To prove Proposition~\ref{proposition:bijection} we need some
preliminaries. 
Let $\mfi, \mfm$ be as above. 
For any $k$
with $1 \leq k \leq n$, we define the notation 
\[
\mfi[k] := (i_k, i_{k+1}, \ldots, i_n)
\quad 
\mfm[k] := (m_k, m_{k+1}, \ldots, m_n)
\]
so $\mfi[k]$ and $\mfm[k]$ are obtained from $\mfi$ and $\mfm$ by
deleting the left-most $k-1$ coordinates. The following lemma is immediate
from the inductive nature of the definitions of the polytopes $P(\mfi,\mfm)$
and of the condition \textbf{(P)}. 

\begin{lemma}\label{lemma:inductive P and condition P}
Let $\mfi = (i_1, \ldots, i_n) \in \{1,\ldots, r\}^n$ be a word and
$\mfm = (m_1, \ldots, m_n) \in \Z^n_{\geq 0}$ be a multiplicity
list.
 \begin{enumerate}
  \item Suppose $(x_1,\ldots,x_n) \in P(\mfi,\mfm)$. For any $k$ with $1 \leq k \leq n-1$, we have 
\[
(x_{k+1}, \ldots, x_n) \in P(\mfi[k+1], \mfm[k+1]).
\]
\item If $(\mfi, \mfm)$ satisfies condition \textbf{(P)}, then for any $k$ with $1 \leq k \leq n-1$ and any
 $(x_{k+1}, \ldots, x_n) \in P(\mfi[k+1],\mfm[k+1])$, the vector 
  $(0,\ldots,0, x_{k+1},\ldots,x_n)$ lies in $P(\mfi,\mfm)$, where
  $(0,\ldots,0, x_{k+1},\ldots,x_n)$ is the vector obtained by adding
  $k$ zeroes to the right. 
\item If $(\mfi, \mfm)$ satisfies condition \textbf{(P)}, then for any $k$
    with $1 \leq k \leq n$, the pair $(\mfi[k], \mfm[k])$ also satisfies
    condition \textbf{(P)}. 
\end{enumerate}
\end{lemma}

To prove Proposition~\ref{proposition:bijection}, the plan is to 
first explicitly
construct a map from $P(\mfi, \mfm) \cap \Z^n$ to $\mc{T}(\mfi, \mfm)$
and then prove that it is a bijection. 
Actually it will be convenient to define a sequence of maps from
$\varphi_k: \Z^k_{\ge 0} \to \Pi \cap \{\mathbf{O}\}$; the map
$\varphi:=\varphi_1$ will be the desired bijection between
$P(\mfi,\mfm) \cap \Z^n$ with $\mc{T}(\mfi, \mfm)$.

\begin{definition}\label{definition:varphi_k}
Let $\mfi = (i_1, \ldots, i_n) \in \{1,\ldots, r\}^n$ be a word and
$\mfm = (m_1, \ldots, m_n) \in \Z^n_{\geq 0}$ be a multiplicity
list.
Let $k$ be an 
integer with $1 \leq k \leq n$. We define a map $\varphi_k: \Z^k_{\geq
  0} \to \Pi \cup \{\mathbf{O}\}$ associated to $\mfi$ and $\mfm$ by 
\begin{equation}\label{eq:varphi_k}
\varphi_k(x_k,\dots,x_n):=f_{\be_k}^{x_k}(\pi^{\la_k} \star
f_{\be_{k+1}}^{x_{k+1}}(\pi^{\la_{k+1}} \star \cdots  \star
f_{\be_n}^{x_n}(\pi^{\la_n})\cdots))
\end{equation}
where \(\la_k:=m_k\varpi_{\be_k}\) for $1 \leq k \leq n$. (Although
the map $\varphi_k$ depends on $\mfi$ and $\mfm$, for simplicity we
omit it from the notation.) 
\end{definition}
From the definition it is immediate that the $\varphi_k$ are related
to one another by the equation 
\[
\varphi_k(x_k, \ldots, x_n) = f_{\be_k}^{x_k}(\pi^{\la_k} \star
\varphi_{k+1}(x_{k+1},\ldots,x_n))
\]
for $1 \leq k<n$. It will be also useful to introduce the notation 
\begin{equation}\label{eq:definition tau_k} 
\tau_k(x_{k+1},\ldots,x_n) := \pi^{\la_k} \star \varphi_{k+1}(x_{k+1}, \ldots, x_n)
\end{equation}
for $1 \leq k <n$ and we set $\tau_n :=
\pi^{\la_n}$, from which it immediately follows that
\begin{equation}\label{eq:varphi in terms of tau}
\varphi_k(x_k, \ldots,x_n) = f_{\be_k}^{x_k}(\tau_k(x_{k+1},\ldots,x_n)).
\end{equation}

With this notation in place we can interpret the functions $A_k$
appearing in the definition of $P(\mfi,\mfm)$ naturally in terms of
paths. Recall that the endpoint $\pi(1)$ of a path $\pi \in \Pi$ is
called its \textbf{weight} and we denote it by $\mr{wt}(\pi):=\pi(1)$. 

\begin{lemma}\label{lemma:endpoints}
Let $(x_1, \ldots, x_n)\in \Z^n_{\geq 0}$ and let $k$ be an integer,
$0 \leq k \leq n-1$. If $\varphi_{k+1}(x_{k+1},\ldots,x_n) \neq
\mathbf{O}$ then 
\[
\mr{wt}(\varphi_{k+1}(x_{k+1},\ldots,x_n)) = m_{k+1}\varpi_{\beta_{k+1}} +
\cdots + m_n \varpi_{\beta_n} - x_{k+1} \beta_{k+1} - \cdots - x_n
\beta_n.
\]
Moreover, if in addition $k\geq 1$ then $\tau_k(x_{k+1},\ldots,x_n) \neq \mathbf{O}$ and 
\[
\mr{wt}(\tau_k(x_{k+1},\ldots,x_n)) = m_k \varpi_{\beta_k} + m_{k+1}\varpi_{\beta_{k+1}} +
\cdots + m_n \varpi_{\beta_n} - x_{k+1} \beta_{k+1} - \cdots - x_n
\beta_n
\]
so in particular 
\begin{equation}\label{eq:Ak as endpoints}
A_k(x_{k+1},\ldots,x_n) = \langle \mr{wt}(\tau_k(x_{k+1},\ldots,x_n)),
\beta_k^\vee \rangle. 
\end{equation}
\end{lemma} 

\begin{proof}
  Under the hypothesis that $\varphi_{k+1}(x_{k+1},\ldots,x_n)$ is an
  honest path (i.e. it is not $\mathbf{O}$), the first statement of
  the lemma is immediate from the definition of $\varphi_k$, Lemma~\ref{lemma:weights}, and
  Lemma~\ref{lem:Littelmann}(1). 
  The other statements of the lemma are then straightforward from the definitions.
\end{proof}

In words, the equation~\eqref{eq:Ak as endpoints} says that the functions $A_k$ measure the pairing of the endpoint of
$\tau_k(x_{k+1},\ldots,x_n)$ against the coroot
$\beta_k^\vee$ (assuming $\tau_k(x_{k+1},\ldots,x_n)$ is an honest
path).

Now we show that when $\varphi_k$ is restricted to the
subset $P(\mfi[k],\mfm[k]) \cap \Z^{n-k+1}$, the output is an honest
path in $\Pi$  (i.e. it is not the formal symbol $\mathbf{O}$). From the definition of standard tableaux it immediately follows that
the output is also in fact an element in $\mc{T}(\mfi[k],\mfm[k])$. 

\begin{lemma}
Let $k$ be an integer with $1 \leq k \leq n$. The map $\varphi_k$
restricts to a map \[\varphi_k: P(\mfi[k],\mfm[k]) \cap \Z^{n-k+1} \to
\mc{T}(\mfi, \mfm).\]
\end{lemma}

\begin{proof}
We first show that the output of the maps $\varphi_k$ are honest paths (i.e. $\neq \mathbf{O}$). 
We argue by induction, and since the definition of the $\varphi_k$ is a
composition of operators starting with $f_{\be_n}$ (not $f_{\be_1}$)
the base case is $k=n$. From the definition of $P(\mfi,\mfm)$
we know that $x_n\le m_n=\lee \pi^{\la_n}(1),\be_n^{\vee}\ree$, so it
suffices to prove that for such $x_n$, we have
$f_{\be_n}^{x_n}(\pi^{\la_n} = \pi^{m_n \varpi_{\beta_n}}) \neq
\mathbf{O}$. Since $\pi^{\la_n}$ is a straight-line path from $0$ to
$\la_n =m_n \varpi_{\beta_n}$, the constants $Q$ and $P$ in the definition of
$f_{\beta_n}$ (applied to $\pi^{\la_n}$) are $0$ and $h_{\be_n}(1) - Q
= \langle m_n \varpi_{\beta_n}, \beta^\vee_n \rangle = m_n$
respectively. Thus by 
Lemma ~\ref{lem:Littelmann}(3), we may conclude 
$\varphi_n(x_n):=f_{\be_n}^{x_n}(\pi^{\la_n})\ne \mathbf{O}$, which 
completes the base case. Now suppose 
that $1 \leq k < n$ and $\varphi_{k+1}(x_{k+1},\ldots,x_n) \neq \mathbf{O}$, which in
turn implies $\tau_{k}(x_{k+1},\ldots,x_n) \neq \mathbf{O}$ since concatenation of paths always results in a path. 
We must show that 
$\varphi_k(x_k, \ldots, x_n) = f_{\be_k}^{x_k}(\tau_k) \neq \mathbf{O}$. Since $\tau_{k}$ 
is a path starting at the origin $0$, the constants 
$Q$ and $P$ in the definition of $f_{\beta_1}$ (applied to $\tau_k(x_{k+1},\ldots,x_n)$) are $\leq
0$ and $\geq \langle \mr{wt}(\tau_k(x_{k+1},\ldots,x_n)), \beta_k^\vee \rangle$ respectively. In
particular, again by Lemma~\ref{lem:Littelmann}(3) it suffices to show
that $x_k \leq \langle \mr{wt}(\tau_k(x_{k+1},\ldots,x_n)),
\beta_k^\vee \rangle$. Since $\tau_k(x_{k+1},\ldots,x_n) \neq
\mathbf{O}$ and $(x_k, \ldots, x_n) \in P(\mfi[k], \mfm[k])$, the result
then holds by definition of $P(\mfi[k],\mfm[k])$ and the
interpretation of the $A_k$ given in
Lemma~\ref{lemma:endpoints}.
It remains to check that the paths
$\varphi_k(x_{k+1},\ldots,x_n) \in \Pi$ are standard tableaux, but this
follows directly from Definition~\ref{definition:standard tableau}.

 \end{proof}

From the above discussion we have a well-defined map 
\begin{equation}\label{eq:def varphi} 
\varphi := \varphi_1: P(\mfi,\mfm) \cap \Z^n \to \mc{T}(\mfi,\mfm).
\end{equation}
We need to prove that 
$\varphi$ is a bijection. For this it is useful to introduce another
condition on $(\mfi,\mfm)$ which we call condition \textbf{(P')}; it
is formulated in terms of the paths $\tau_k$ and the raising operators
$e_{\beta_k}$.

\begin{definition}
Let $\mfi = (i_1, \ldots, i_n) \in \{1,\ldots,r\}^n$ be a word and
$\mfm = (m_1, \ldots, m_n) \in \Z^n_{\geq 0}$ be a multiplicity
list. We say that the pair $(\mfi,\mfm)$ \textbf{satisfies condition
  \textbf{(P')}} if for all $(x_1, \ldots, x_n) \in P(\mfi, \mfm) \cap
\Z^n_{\geq 0}$ and
all $k$ with $1 \leq k \leq n$, we have $e_{\beta_k}(\tau_k(x_{k+1},\ldots,x_n)) = \mathbf{O}$. 
\end{definition}

It may be conceptually helpful to note that, from our interpretation
of the functions $A_k$ in Lemma~\ref{lemma:endpoints} and the
definitions of $P(\mfi,\mfm)$ and $\tau_k$, we may think of condition 
\textbf{(P)} as saying that the \emph{endpoints} of certain paths $\tau_k$ are always
contained in the affine half-space defined by $\{\langle \cdot, \beta_k
\rangle \geq 0\}$ (i.e. the half-space pairing non-negatively against
the coroot $\beta^\vee_k$). Moreover, from
Lemma~\ref{lem:Littelmann}(3) we see that in order to show
$e_{\beta_k}(\tau_k)=\mathbf{O}$ for a given path $\tau_k$, it suffices to show that the
\emph{entire path}
$\tau_k$ lies in the same affine half-space. Thus, roughly speaking,
condition \textbf{(P)} is about endpoints, whereas condition
\textbf{(P')} is about the entire path. 

\begin{remark}\label{remark:P' does not imply P}
From the above discussion it may seem that condition \textbf{(P')} is
stronger than condition \textbf{(P)}. This is not the case. 
For instance, for $G=SL(3,\C)$, the
pair $\mfi=(1,2,1)$ and $\mfm=(0,1,1)$ is an example where $(\mfi,
\mfm)$ satisfies condition \textbf{(P')} but it does not satisfy
condition \textbf{(P)}; for an illustration of the corresponding
polytope $P(\mfi,\mfm)$ see Example~\ref{example:nonexample}. The subtlety is in the quantifiers in the
definitions of the conditions \textbf{(P)} and \textbf{(P')}. 
\end{remark}

\begin{proposition}\label{prop:P implies P'}
Let $\mfi = (i_1, \ldots, i_n) \in \{1,\ldots,r\}^n$ be a word and
$\mfm = (m_1, \ldots, m_n) \in \Z^n_{\geq 0}$ be a multiplicity
list.
If the pair 
$(\mfi, \mfm)$ 
satisfies condition \textbf{(P)} then $(\mfi, \mfm)$ satisfies
condition \textbf{(P')}. 
\end{proposition}

Since condition \textbf{(P')} is phrased in terms of the $e_{\be_k}$
and because the raising and lowering operators acts as inverses
(provided the composition makes sense) as in
Lemma~\ref{lem:Littelmann}(2), once we know
Proposition~\ref{prop:P implies P'} it is straightforward to show that
$\varphi$ is a bijection. Indeed, we suspect that the argument given
below is standard for the experts, but we include it for
completeness.

\begin{proof}[Proof of Proposition~\ref{proposition:bijection} (assuming Proposition~\ref{prop:P implies P'})]

By Proposition~\ref{prop:P implies P'} we may assume that condition \textbf{(P')} holds. 
First we prove by induction that $\varphi_k$ is injective for each
$k$, starting with the base case $k=n$. Suppose
\begin{equation}\label{eq:xn yn}
\varphi_n(x_n) = f^{x_n}_{\beta_n}(\pi^{\lambda_n}) =
f^{y_n}_{\beta_n}(\pi^{\lambda_n})=\varphi_n(y_n)
\end{equation}
and also suppose for a
contradiction that $x_n<y_n$. Applying $e_{\beta_n}^{x_n+1}$ to the
LHS of~\eqref{eq:xn yn} yields $e_{\beta_n}(\pi^{\lambda_n})$ since by
Lemma~\ref{lem:Littelmann}(2) we know $e_{\beta_n}$ is
inverse to $f_{\beta_n}$ whenever the image of $f_{\beta_n}$ is $\neq \mathbf{O}$.
By condition \textbf{(P')}, $e_{\beta_n}(\pi^{\lambda_n}) =
e_{\beta_n}(\tau_n) = \mathbf{O}$. On the other hand, applying
$e_{\beta_n}^{x_n+1}$ to the RHS of~\eqref{eq:xn yn} yields
$f_{\beta_n}^{y_n-x_n-1}(\pi^{\lambda_n})$ which is $\neq \mathbf{O}$
since $y_n-x_n-1\geq 0$ by assumption. This contradicts~\eqref{eq:xn
  yn} and so $x_n=y_n$ and we conclude $\varphi_n$ is injective. This
completes the base case. Now
suppose by induction that $\varphi_{k+1}$ is injective; we need to show $\varphi_k$
is injective. Assume 
\[
\varphi_k(x_k, \ldots, x_n) =
f^{x_k}_{\beta_k}(\tau_k(x_{k+1},\ldots,x_n)) =
f^{y_k}_{\beta_k}(\tau_k(y_{k+1},\ldots,y_n)) = \varphi_k(y_k, \ldots,
y_n).
\]
and suppose also that $x_k<y_k$. The same argument
as above, namely applying $e_{\beta_k}^{x_k+1}$ to both sides,
yields a contradiction due to the condition \textbf{(P')}. Thus
$x_k=y_k$. Applying $e_{\beta_k}^{x_k}$ to both sides of the equation
above we obtain $\tau_k(x_{k+1},\ldots, x_n) =
\tau_k(y_{k+1},\ldots,y_n)$. Concatenation by $\pi^{\lambda_k}$ is
evidently injective, so
$\varphi_{k+1}(x_{k+1},\ldots,x_n)=\varphi_{k+1}(y_{k+1},\ldots,y_n)$,
but then by the inductive assumption we have
$(x_{k+1},\ldots,x_n)=(y_{k+1},\ldots,y_n)$. This proves
$(x_k,\ldots,x_n)=(y_k,\ldots,y_n)$ and hence that $\varphi_k$ is
injective as desired.

Now we claim $\varphi_k$ is surjective for each $k$. 
We argue by induction on the size of $n$. First
  consider the base case $n=1$, so $w = (\beta_1 = \beta)$,
  $m=(m_1=m)$, and $P(w,m) = [0,m]$. 
 By definition, a standard tableau of shape $\la=m\varpi_{\be}$ is
  of the form $f_{\be}^x(\pi^{\la})$ for some $x\in\Z_{\ge 0}$. Since
  $\pi^{\la}$ is a straight-line path from $0$ to $m \beta$, the constants
  $Q$ and $P$ in the definition of $f_\be$ applied to $\pi^\la$ are $0$ and $m$
 respectively. Then for $x$ a non-negative integer we know by
 Lemma~\ref{lem:Littelmann}(3) that 
 $f^x_\be(\pi^{\la}) \neq \mathbf{O}$ 
 if and only if $x\le m$. 
 Since $P(\mfi,\mfm)=[0,m]$ in this
 case, we conclude that $\varphi_1$ is surjective if $n=1$, as desired.

Now assume by induction that each $\varphi_k$ is surjective (hence
bijective) for words of length $<n$. From Lemma~\ref{lemma:inductive P and
  condition P}(3) we know that $(\mfi[k],\mfm[k])$ satisfies condition
\textbf{(P)} (and hence condition \textbf{(P')}). By the inductive
assumption we may therefore assume 
that $\varphi_k: P(\mfi[k],\mfm[k]) \cap
\Z^{n-k+1} \to \mc{T}(\mfi[k],\mfm[k])$ is a bijection for $k>1$ and we
wish to show $\varphi=\varphi_1$ is surjective. 
By definition of the standard tableaux,
any element in $\mc{T}(\mfi,\mfm)$ is of the
form $f_{\beta_1}^{\ell_1}(\pi^{\la_1} \star \tau')$
for some $\tau' \in \mc{T}(w[2], m[2])$ and some $\ell_1 \in \Z_{\geq
  0}$. By the inductive assumption, we know that there exists some
$(x_2, \ldots, x_n) \in P(\mfi[2], \mfm[2])$ such that $\tau' = \varphi_2(x_2,
\ldots, x_n)$. From the definition of $P(\mfi,\mfm)$, in 
order to prove the surjectivity it would suffice
to show that 
\[
f_{\beta_1}^{\ell_1}(\pi^{m_1 \varpi_{\beta_1}} \star \varphi_2(x_2, \ldots,
x_n)) = f_{\beta_1}(\tau_1(x_2,\ldots,x_n)) \neq \mathbf{O} \Rightarrow \ell_1 \leq A_1(x_2, \ldots, x_n).
\]
From Lemma~\ref{lem:Littelmann}(3) we know 
$f_{\beta_1}^{\ell_1}(\tau_1) \neq \mathbf{O} \Leftrightarrow \ell_1
\leq P$
where $P$ is defined to be the integral part of $\langle
\mr{wt}(\tau_1(x_2,\ldots,x_n)), \beta_1^\vee \rangle - Q$ and 
\(Q =  
\min_{t \in [0,1]} \langle
\tau_1(x_2,\ldots,x_n)(t),\beta_1^\vee \rangle.\) 
Since
$\tau_1(x_2,\ldots,x_n) \neq \mathbf{O}$ by assumption,
we know from~\eqref{eq:Ak as endpoints} that $A_1(x_2,\ldots,x_n)=\langle
\mr{wt}(\tau_1(x_2,\ldots,x_n)), \beta_1^\vee \rangle$ and it is
evident from the definition of $A_1$ that for $(x_2,\ldots,x_n) \in
\Z^{n-1}$, the value $A_1(x_2, \ldots,x_n)$ is integral. Hence 
it suffices to show 
that $Q=0$, and again from Lemma~\ref{lem:Littelmann}(3) 
this is equivalent to showing that
$e_{\beta_1}(\tau_k(x_2,\ldots,x_n)) = \mathbf{O}$. Note that the vector
$(0,x_2,\ldots,x_n)$ lies in $P(\mfi,\mfm)$ by
Lemma~\ref{lemma:inductive P and condition P}(2). By applying the
statement of condition \textbf{(P')} to $(0,x_2,\ldots,x_n)$ and $k=1$
we obtain that $e_{\beta_1}(\tau_k(x_2,\ldots,x_n))=\mathbf{O}$ as
desired. 
This completes the proof. 
\end{proof}

It remains to justify Proposition~\ref{prop:P implies P'}. 
The following simple lemma will be
helpful.

\begin{lemma}\label{lemma:concatenation linear comb}
  Let $\pi \in \Pi$ be a piecewise linear path in $X_\R$. 
\begin{enumerate} 
\item Let $\pi^{\lambda}$ be
  a linear path for some $\lambda \in X_\R$. Then 
 for any $t \in [0,1]$, there exist non-negative real 
  constants $a, c \geq 0$ and $s \in [0,1]$ such that
  \((\pi^{\lambda} \star \pi)(t) = a \lambda + c \pi(s).\) 
\item  Let 
$\beta$ be a 
  simple root. Let $x$ be a positive integer and assume that $f^x_\beta(\pi) \neq \mathbf{O}$. Then 
for any $t \in [0,1]$, 
there exists $b \in \R$ with $0 \leq b \leq x$ such that 
  \(f^x_\beta(\pi)(t) = \pi(t) + b (-\beta)\) 
where $0 \leq b \leq x$. 
\item Let $\pi \in \Pi$ be a path in $X_\R$. Let $\{\beta_1, \ldots,
  \beta_j\}$ 
  be any sequence of simple roots and $n_1,
  \ldots, n_j\in \Z_{\geq 0}$ any sequence of non-negative
  integers. Then any point along the path
  $f^{n_1}_{\beta_1}(\pi^{n_1 \varpi_{\beta_1}} \star
  f^{n_2}_{\beta_2}(\pi^{n_2 \varpi_{\beta_2}} \star \cdots
  \star f^{n_j}_{\beta_j}(\pi^{n_j\varpi_{\beta_j}}\star \pi) \cdots ))$ can be expressed as a linear combination 
\[
\sum_{\ell=1}^{j} a_\ell n_\ell \varpi_{\alpha_\ell} + \sum_{\ell=1}^{j}
b_\ell (-\beta_\ell) + c \pi(s)
\]
for some $a_\ell, b_\ell, c \geq 0$ non-negative real
constants and some $s \in [0,1]$. 
\item  Let $\mfi=(i_1, \ldots, i_n) \in \{1,2,\ldots,r\}^n$ and $\mfm=(m_1, \ldots, m_n) \in \Z^n_{\geq 0}$ be
  a word and multiplicity list and let $k$ be an integer with $1 \leq k
  \leq n$. Let
  $\varphi_k$ denote the map associated to $\mfi, \mfm$ as in Definition~\ref{definition:varphi_k}. Then any point along the path
  $\varphi_k(x_k,\ldots, x_n)$ can be expressed as a linear combination 
\begin{equation}\label{eq:cj dj} 
\sum_{\ell=k}^n a_\ell m_\ell \varpi_{\beta_\ell} + \sum_{\ell=k}^n b_\ell (-\beta_\ell)
\end{equation}
where $a_\ell, b_\ell \geq 0$. 
\end{enumerate} 
\end{lemma}

\begin{proof} 
First we prove (1). From the definition~\eqref{eq:concatenation} of paths and the
definition of a straight-line path $\pi^\lambda$ it follows that for
$t \in [0,\frac{1}{2}]$ we may take $a=2t$ and $c=0$, since
$(\pi^\lambda \star \pi)(t) = \pi^\lambda(2t) = 2t \lambda$ in this
case. On the other hand, if $t \in [\frac{1}{2},1]$ then we may take
$a=1, c=1$ and $s = 2t-1$, since by~\eqref{eq:concatenation}
we have $(\pi^\lambda \star \pi)(t) := \pi^\lambda(1) + \pi(2t-1) =
\lambda + \pi(2t-1)$. This proves the claim. 

Next we prove (2). Recall that the reflection operator $s_\beta$ acts by
$s_\beta(\alpha):= \alpha - \langle \alpha, \beta^\vee \rangle \beta$,
so for any path $\pi$ and for any time $t$ we have $s_\beta(\pi)(t) :=
s_\beta(\pi(t)) = \pi(t) - \langle \pi(t), \beta^\vee \rangle \beta =
\pi(t) + \langle \pi(t), \beta^\vee \rangle (-\beta)$ and in
particular, $s_\beta(\pi)(t)$ is a linear combination of $\pi(t)$ and
$-\beta$. Additionally, from Definition~\ref{definition:f_beta} we
know that $f_\beta(\pi) := \pi_1 \star s_\beta(\pi_2) \star \pi_3$
where $\pi_1, \pi_2$ and $\pi_3$ are defined in~\eqref{eq:definitions
  of pi_1 pi_2 pi_3} and from the discussion preceding
Definition~\ref{definition:f_beta} which defines $p$ and $x$ iti
follows that $\langle \pi_2(t), \beta^\vee \rangle \in [0,1]$ for all
$t$. 
To prove the claim we begin with the base case $x=1$. 
Consider each of the $3$ components of $f_\beta(\pi)$ in turn. For the first
portion of the path (corresponding to $\pi_1$), the operator $f_\beta$
does not alter the path at all, so for such $t$ we have
$f_\beta(\pi)(t) = \pi(t)$ and the claim of the lemma holds with
$b=0$. For $t$ in the second portion of the path, we have $\pi(t)
= \pi_1(p)+\pi_2(t')$ (here $t'$ is determined by $t$ by some
reparametrization coming from the concatenation operation) and
$f_\beta(\pi)(t) = \pi_1(p)+s_\beta(\pi_2(t')) =
\pi_1(p)+\pi_2(t')+\langle \pi_2(t'), \beta^\vee \rangle (-\beta) =
\pi(t) + \langle \pi_2(t'), \beta^\vee \rangle (-\beta)$. As we have
already seen, $\langle \pi_2(t'), \beta^\vee\rangle \in [0,1]$, so
choosing $b = \langle \pi_2(t'), \beta^\vee \rangle$ does the
job. Finally, again from the discussion preceding the definitions of
$\pi_1, \pi_2$ and $\pi_3$ it follows that $\langle \pi_2(1),
\beta^\vee\rangle = 1$ so for the last (third) portion of the path we
have that $f_\beta(\pi)(t) = (\pi(x)-\beta)+\pi_3(t'') =
(\pi(x)+\pi_3(t'') - \beta = \pi(t)-\beta$ where again $t''$ is
determined by $t$ by a reparametrization. By choosing $b=1$ we
see that the claim holds in this case also. 
Applying the same argument $x$ times yields the result. 

The statements (3) and (4) follow straightforwardly by applying (1)
and (2) repeatedly. 
\end{proof}

The following elementary observation is also conceptually useful. 
For two simple positive roots $\alpha, \beta$, we say that $\alpha$
and $\beta$ are \textbf{adjacent} if they are distinct and they
correspond to two adjacent nodes in the corresponding Dynkin
diagram. (From properties of the Cartan matrix, $\alpha$ and $\beta$
are adjacent precisely when the value of the pairing $\langle \alpha, \beta^\vee
\rangle$ is strictly negative.)  Then it is immediate that  
$A_k(x_{k+1}, \ldots, x_n)$ can be interpreted as
\begin{equation}\label{eq:Ak} 
A_k(x_{k+1}, \ldots, x_n) = m_k + \left( \sum_{\substack{j>k \\
      \beta_j = \beta_k}} (m_j - 2 x_j) \right) - \left(
  \sum_{\substack{j > k \\ \beta_j \textup{
        adjacent to } \beta_k}} x_j \langle \beta_j, \beta_k^\vee
  \rangle \right).
\end{equation}

\begin{proof}[Proof of Proposition~\ref{prop:P implies P'}]

We begin by noting that the path $\tau_n$ is by definition
$\pi^{\lambda_n}$ where $\lambda_n := m_n \be_n$. Thus $Q=0$ in this
case and by Lemma~\ref{lem:Littelmann}(3) we conclude
$e_{\be_n}(\tau_n)=\mathbf{O}$. 
So it remains to check the cases $k<n$. 
As in the discussion above, by Lemma~\ref{lem:Littelmann}(3) and by
the definition of the raising operators, in order to prove the claim 
it suffices to prove that for any 
$(x_1, \ldots, x_n) \in
    P(\mfi, \mfm)$ and any $k$ with $1 \leq k \leq n-1$, we have 
    \begin{equation}
      \label{eq:minimum pairing}
      \min_{t\in [0,1]} \{ \langle \tau_k(x_{k+1},\ldots,x_n)(t), \beta_k^\vee \rangle \}
      \geq 0 
    \end{equation}
which is equivalent to 
\begin{equation}
  \label{eq:minimum pairing m_k}
   \min_{t \in [0,1]} \{ \langle \varphi_{k+1}(x_{k+1}, \ldots, x_n)(t),
   \beta_k^\vee \rangle \geq - m_k
\end{equation}
by definition of the $\tau_k$ and $\varphi_k$. 

We use induction on the size of $n$. We already proved the case $n=1$
above so the 
base case is $n=2$ and $k=1$.  Let $\mfi =(i_1, i_2)$ with associated sequence of
simple roots $(\beta_1,
\beta_2)$ and $\mfm= (m_1, m_2)$. Let $(x_1, x_2) \in P(\mfi,\mfm)$. 
Then we have 
$0 \leq x_2 \leq m_2$ so an explicit computation shows 
$\varphi_2(x_2) = f_{\beta_2}^{x_2}(\pi^{m_2 \varpi_{\beta_2}}) =
\pi^{x_2(\varpi_{\beta_2}-\beta_2)} \star \pi^{(m_2 -
  x_2)\varpi_{\beta_2}}.$
Hence we wish to show that 
\[
\min_{t \in [0,1]} \{ \langle \varphi_2(x_2)= \pi^{x_2(\varpi_{\beta_2} -\beta_2)}
\star \pi^{(m_2-x_2)\varpi_{\beta_2}}(t), \beta_1^\vee \rangle \} \geq
-m_1.
\]
First consider the case $\beta_1
\neq \beta_2$. Since $\langle \varpi_{\beta_2}, \beta_1^\vee
\rangle = 0$ and $\langle \beta_2, \beta_1^\vee \rangle \leq 0$ for
any two distinct simple roots, and $x_2 \geq 0$ by assumption, we can
see that 
$\langle \pi^{x_2(\varpi_{\beta_2} -\beta_2)}
\star \pi^{(m_2-x_2)\varpi_{\beta_2}}(t), \beta_1^\vee \rangle \geq 0$
for all $t$. In particular the minimum value taken over all $t$ is
$0$, which is greater than or equal to $-m_1$ as desired (since $m_1 \geq 0$ by
assumption). 
Next consider the case $\beta_1 = \beta_2$. In this case, the
inequalities defining $P(\mfi,\mfm)$ are 
\[
0 \leq x_2 \leq m_2 \textup{ and }  0 \leq x_1 \leq \langle m_1
\varpi_{\beta_1} + m_2 \varpi_{\beta_2} - x_2 \beta_2, \beta_1^\vee
\rangle = m_1 + m_2 - 2x_2.
\]
From the condition \textbf{(P)}, for any choice of $x_2$ with $0 \leq
x_2 \leq m_2$ we must have $A_1(x_2) = m_1 +m_2 - 2x_2 \geq 0$. In
particular, for $x_2 = m_2$ we must have $m_1 - m_2 \geq 0$, from
which it follows $m_1 \geq m_2$. Next notice that, since the vector
$(m_2 - x_2)\varpi_{\beta_2}$ pairs non-negatively with $\beta_1^\vee
= \beta_2^\vee$, the minimum value of the function 
\[
t \mapsto \langle \pi^{x_2(\varpi_{\beta_2} -\beta_2)}
\star \pi^{(m_2-x_2)\varpi_{\beta_2}}(t), \beta_1^\vee \rangle
\]
occurs at the endpoint of $\pi^{x_2(\varpi_{\beta_2} - \beta_2)}$
where the value is $-x_2$. From the assumptions we know $x_2 \leq
m_2$, so $-x_2 \geq -m_2$. Also from the above we have seen that $m_1
\geq m_2$, so $-m_2 \geq - m_1$ and finally we obtain $-x_2 \geq
-m_1$. This completes the base case. 

We now assume by induction that the statement of the proposition holds for
words and multiplicity lists of length $\leq n-1$ and we must prove
the statement for $n$. As above, we already know the statement holds for $k=n$. Next suppose $1 < k < n$. By Lemma~\ref{lemma:inductive P and condition P} we know that $(\mfi[k], \mfm[k])$ satisfies condition \textbf{(P)} and $(x_k, \ldots, x_n)$ lies in $P(\mfi[k],\mfm[k])$. Since $\mfi[k],\mfm[k]$ have length strictly less than $n$, by the inductive assumption we know the statement holds for such $k$. Thus it remains to check the case $k=1$, i.e. that $e_{\beta_1}(\tau_1(x_2, \ldots, x_n)) = \mathbf{O}$ for $(x_1, \ldots, x_n) \in P(\mfi,\mfm)$. 
First consider the case in which the simple root $\beta_1$ does not
appear in the word $(\beta_2, \ldots, \beta_n)$. By
Lemma~\ref{lemma:concatenation linear comb}(4), any
point along the path $\varphi_2(x_2, \ldots,
x_n)$ can be written in the form 
$\sum_{\ell=2}^n a_\ell \varpi_{\beta_\ell} + \sum_{\ell=2}^n
b_\ell  (-\beta_\ell)$
where $a_\ell, b_\ell \geq 0$ are non-negative real constants,
and all the simple roots $\beta_\ell$ are distinct from
$\beta_1$. Then for any time $t$ we have 
$\langle \varphi_2(x_2, \ldots, x_n)(t), \beta_1^\vee \rangle  = \left\langle    \sum_{\ell=2}^n a_\ell \varpi_{\beta_\ell} + \sum_{\ell_2}^n
b_\ell (-\beta_\ell), \beta_1^\vee \right\rangle 
 = \left\langle - \sum_{\ell=2}^n b_\ell \beta_\ell, \beta_1^\vee
 \right\rangle 
 = - \sum_{\ell=2}^n b_\ell \langle \beta_\ell, \beta_1^\vee \rangle
 \geq 0$
where the second equality is because $\langle \varpi_{\beta_{\ell}},
\beta_1^\vee \rangle = 0$ for $\beta_\ell \neq \beta_1$ and the last
inequality is because $\langle \beta_\ell, \beta_1^\vee \rangle \leq
0$ for $\beta_\ell \neq \beta_1$.
Since $m_1 \geq 0$ by assumption, we conclude that $\langle \varphi(x_2, \ldots,
x_n)(t), \beta_1^\vee \rangle \geq 0 \geq -m_1$ for all $t$, which
yields the desired result.

Next we consider the case when $\beta_1$ occurs in the sequence
$(\beta_2,\ldots,\beta_n)$. 
Let $s$ be the smallest index with $s \geq 2$ such that $\beta_s =
\beta_1$, i.e., it is the first place after $\beta_1$ where the
repetition occurs. Since the length of $\mfi[s]$ 
is $n-1$, from the inductive assumption we know that 
$\min_{t \in [0,1]} \{ \langle \tau_s(x_{s+1},\ldots,x_n)(t), \beta_s^\vee = \beta_1^\vee
\rangle \} \geq 0$.
Note also that the path $\tau_s$ has the property that the minimum value 
$\min_{t \in [0,1]} \{ \langle \tau_s(x_{s+1},\ldots,x_n)(t), \beta_s^\vee = \beta_1^\vee
\rangle \}$ as well as the endpoint pairing $\langle
\mr{wt}(\tau_s),\beta_s^\vee \rangle$ are both integers; this follows from its
construction. Also by definition, the operator $f_{\beta_s}$ preserves these
properties; moreover, for such a path $\tau'$ it follows from the
definition of $f_{\beta_s}$ that 
$\min_{t \in [0,1]} \{ \langle f_{\beta_s}(\tau')(t), \beta_s^\vee
\rangle \} = 
\min_{t \in [0,1]} \{ \langle \tau'(t), \beta_s^\vee
\rangle \} -1$, i.e., the minimum decreases by precisely $1$. 
From this we conclude that $\varphi_s(x_s, \ldots, x_n)=
f^{x_s}_{\beta_s=\beta_1}(\tau_s)$ satisfies
\begin{equation}\label{eq:greater than -x_s}
\langle \varphi_s(x_s, \ldots, x_n)(t), \beta_1^\vee = \beta_s^\vee \rangle
\geq -x_s \textup{ for all } t \in [0,1].
\end{equation}
By definition $\varphi_2(x_2, \ldots, x_n)$ is obtained from $\varphi_s(x_s,
\ldots, x_n)$ 
by 
\[
\varphi_2(x_2, \ldots, x_n) := f_{\beta_2}^{x_2}(\pi^{m_2 \varpi_{\beta_2}}
\star ( \cdots f_{\beta_{s-1}}^{x_{s-1}}(\pi^{m_{s-1}
  \varpi_{\beta_{s-1}}} \star \varphi_s(x_s, \ldots, x_n)) \cdots ).
\]
By assumption, $\beta_1$ is distinct from all the roots $\beta_\ell$
for $2 \leq \ell \leq s-1$. 
Thus $\langle \varpi_{\beta_\ell}, \beta_1^\vee \rangle=0$ and $\langle -\beta_\ell, \beta_1^\vee \rangle \geq 0$ for $2 \leq \ell \leq s-1$ and 
from Lemma~\ref{lemma:concatenation linear comb}(3) it follows that 
\[
\min_{t \in [0,1]} \{ \langle \varphi_2(x_2, \ldots, x_n)(t), \beta_1^\vee \rangle\} \geq 
\min_{t \in [0,1]} \{ \langle \varphi_s(x_s, \ldots, x_n)(t), \beta_1^\vee \rangle \}.
\]
Since we know from~\eqref{eq:greater than -x_s} that the RHS above
is $\geq -x_s$, it now suffices to
prove that $x_s \leq m_1$. 
Since $(x_1, \ldots, x_n) \in P(\mfi,\mfm)$, we know $(y_s, x_{s+1}, \ldots,
x_n) \in P(\mfi[s], \mfm[s])$ if $0 \leq y_s \leq A_s(x_{s+1}, \ldots,
x_n)$. Also since $(\mfi,\mfm)$ satisfies condition \textbf{(P)}, from
Lemma~\ref{lemma:inductive P and condition P}(2) we know that 
$(y_2, \ldots, y_n) \in 
P(\mfi[2], \mfm[2])$, where 
$y_2=\cdots=y_{s-1}=0$,
$y_s = A_s(x_{s+1}, \ldots, x_n)$, and $y_k=x_k$ for $k\geq s+1$. Then from the condition \textbf{(P)} we conclude that 
\begin{equation*}
   \begin{split}
     A_1(y_2, \ldots, y_n) & = m_1 + (m_s - 2y_s) + 
 \left( \sum_{\substack{k>s \\
       \beta_k = \beta_1=\beta_s}} (m_k - 2 x_k) \right) - \left(
   \sum_{\substack{k > s \\ \beta_k \textup{
         adjacent to } \beta_1=\beta_s}} x_k \langle \beta_k, \beta_1^\vee=\beta_s^\vee
   \rangle \right) \\ 
  & = m_1 + A_s(x_{s+1}, \ldots, x_n) - 2 y_s = m_1 - A_s(x_{s+1},
  \ldots, x_n) \geq 0. 
   \end{split}
 \end{equation*}
or in other words $m_1 \geq A_s(x_{s+1}, \ldots, x_n)$. But the
original $x_s$ was required to satisfy the inequality $x_s \leq
A_1(x_{s+1}, \ldots, x_n)$, from which it follows that $x_s \leq m_1$
as was to be shown. This completes the inductive argument and
hence the proof. 
\end{proof}

\section{Newton-Okounkov bodies of Bott-Samelson varieties}\label{section:NOBY}

The main result of this manuscript is Theorem~\ref{theorem:main},
which gives an explicit description of the Newton-Okounkov body of
$(Z_{\mfi}, L_{\mfi,\mfm})$ with respect to a certain geometric
valuation (to be described in detail below), provided that the word
$\mfi$ corresponds to a reduced word decomposition and
the pair $(\mfi,\mfm)$ satisfies condition \textbf{(P)}.

We first very briefly recall the ingredients in the definition of a
Newton-Okounkov body. For details we refer the reader to
\cite{KavehKhovanskii2012, LazarsfeldMustata2009}. We
begin with the definition of a valuation (in our setting). 

\begin{definition}\label{definition:valuation}
Let $A$ be a $\C$-algebra and $\Gamma$ a totally ordered set with
order $<$. We say that a function $\nu: A \setminus \{0\} \to \Gamma$
is a \textbf{valuation} if 
\begin{enumerate}
\item $\nu(fg)=\nu(f) + \nu(g)$ for all $f,g \in A \setminus \{0\}$, 
\item $\nu(f+g) \geq \min\{\nu(f),\nu(g)\}$ for all $f, g \in A
  \setminus \{0\}$ with $f+g \neq 0$,
\item $\nu(cf) = \nu(f)$ for all $f \in A \setminus \{0\}$ and $c \in
  \C^*$. 
\end{enumerate}
The image in $\Gamma$ of $A \setminus \{0\}$ is clearly a semigroup
and is called the \textbf{value semigroup} of the pair
$(A,\nu)$. Moreover, if in addition the valuation has the property
that 
\begin{enumerate}
\item[(iv)] if $\nu(f)=\nu(g)$, then there exists a non-zero constant
  $\lambda \neq 0 \in \C$ such that $\nu(g-\lambda f)>\nu(g)$ or else
  $g-\lambda f = 0$
\end{enumerate} 
then we say the valuation has \textit{one-dimensional leaves}. 
\end{definition}

In the construction of Newton-Okounkov bodies, we consider valuations
on rings of sections of line bundles. More specifically, let $X$ be a
complex-$n$-dimensional algebraic variety over $\C$,
equipped with a line bundle $L = \mathcal{O}_X(D)$ for some (Cartier)
divisor $D$. 
Consider the corresponding (graded) $\C$-algebra of sections $R =R(L) :=
\oplus_{k \geq 0} R_k$ where $R_k := H^0(X, L^{\otimes k})$. We
now describe a way to geometrically construct a valuation. (Not all valuations
arise in this manner but this suffices for our purposes.) Suppose given a flag
\[
Y_{\bullet}: X = Y_0 \supseteq Y_1 \supseteq \cdots \supseteq Y_{n-1}
\supseteq Y_n = \{\textup{pt}\}
\]
of irreducible subvarieties of $X$ where $\codim_\C(Y_\ell) = \ell$
and each $Y_\ell$ is non-singular at the point
$Y_n=\{\textup{pt}\}$. Such a flag defines a valuation $\nu_{Y_\bullet}:
H^0(X,L) \setminus \{0\} \to \Z^n$ by an inductive procedure involving
restricting sections to each subvariety and considering its order
of vanishing along the next (smaller) subvariety, as follows. 
We will assume that all $Y_i$ are smooth (though this is
not necessary, cf. \cite{LazarsfeldMustata2009}). 
Given a non-zero section $s \in H^0(X, L=\mathcal{O}_X(D))$,
we define 
\[
\nu_1 := \mathrm{ord}_{Y_1}(s)
\]
i.e. the order of vanishing of $s$ along $Y_1$. By choosing a local
equation for $Y_1$ in $X$, we can construct a section $\tilde{s}_1 \in
H^0(X, \mathcal{O}_X(D - \nu_1 Y_1))$ that does not vanish identically
on $Y_1$. By restricting we obtain a non-zero section $s_1 \in
H^0(Y_1, \mathcal{O}_{Y_1}(D - \nu_1 Y_1))$, and define $\nu_2 :=
\mathrm{ord}_{Y_2}(s_1)$. We define each $\nu_i$ by proceeding inductively in the same fashion. 
It is not difficult to see that $\nu_{Y_\bullet}$ thus defined gives a 
valuation with one-dimensional leaves on each $R_k$.

Given such a valuation $\nu$, we may then define 
\[
S(R) = S(R, \nu) := \bigcup_{k >0} \{(k, \nu(\sigma)) \mid
\sigma \in R_k \setminus \{0\} \} \subset \N \times \Z^n
\]
(cf. also \cite[Definition 1.6]{LazarsfeldMustata2009}, where the notation slightly differs)
which can be seen to be an additive semigroup. Now define
$C(R) \subseteq \R \times \R^n$ to be the cone generated by the
semigroup $S(R)$, i.e., it is the smallest closed convex cone centered
at the origin containing $S(R)$. We can now define the central object
of interest. 

\begin{definition}\label{definition:NO}
  Let $\Delta=\Delta(R)=\Delta(R,\nu)$ be the slice of the cone $C(R)$ at $\ell=1$
  projected to $\R^n$ via the projection to the second factor $\R
  \times \R^n \to \R^n$. In other words 
\[
\Delta = \overline{ \textup{conv} \left( \bigcup_{k>0}
    \left\{\frac{x}{k} :  (k,x) \in S(R) \right\} \right) }.
\]
The convex body $\Delta$ is called the \textbf{Newton-Okounkov body of
  $R$} with respect to the valuation $\nu$. 
\end{definition}

In the current manuscript, the geometric objects under study are the
Bott-Samelson variety $Z_\mfi$ and the line bundle $L_{\mfi,\mfm}$
over it. Following the notation above, we wish to study the
Newton-Okounkov body of $R(L_{\mfi,\mfm})=\oplus_{k >0} H^0(Z_\mfi,
L_{\mfi,\mfm}^{\otimes k})$. 
We begin with a description of the flag $Y_\bullet$ of subvarieties with respect
to which we will define a valuation. 
Given $\ell$ with $1 \leq \ell \leq n$, we define a subvariety
$Y_\ell$ of $Z_\mfi$ of codimension $\ell$ by 
\[
Y_\ell := \{ [(p_1, \ldots, p_n)] : p_s=e \textup{ for the last
  $\ell$ coordinates, i.e. for } n-\ell+1 \leq s \leq n \}. 
\]
The subvariety $Y_\ell$ is smooth, since it is evidently isomorphic to
the Bott-Samelson variety $Z_{(i_1, \ldots, i_{n-\ell})}$. In Kaveh's
work on Newton-Okounkov bodies and crystal bases \cite{Kaveh2011}, he
introduces a set of coordinates, which he denotes $(t_1, \ldots, t_n)$, near the point $Y_0 =
\{[(e,e,\ldots,e)]\}$. Near $Y_0$, 
our flag $Y_\bullet$ can be described using Kaveh's coordinates as 
\[
\{t_n=0\} \supset \{t_n=t_{n-1}=0\} \supset \cdots \supset
\{t_n=\cdots=t_2=0\} \supset \{(0,0,\ldots,0)\}.
\]

\begin{remark}\label{remark:valuations}
In particular, with respect to Kaveh's coordinates, our geometric
valuation $\nu_{Y_\bullet}$ is the lowest-term valuation on polynomials
in $t_1, \ldots, t_n$ with respect to the lexicographic order with
$t_1<t_2<\cdots<t_n$. Thus our valuation is different from the
valuation used by Kaveh in \cite{Kaveh2011} and Fujita in \cite{Fujita2015}, since
they take the highest-term valuation with respect to the lexicographic
order with the variables in the reverse order, $t_1>t_2>\cdots>t_n$. In general, it seems to be a rather subtle problem to understand the dependence of the Newton-Okounkov body on the choice of valuation, cf. for instance the discussion in \cite[Remark 2.3]{Kaveh2011}.
\end{remark}

We now state the main theorem of this section, which
is also the main result of this manuscript. Let $P(\mfi,\mfm)$ denote
the polytope of Definition~\ref{def:polytope}.
In the statement below, 
$P(\mfi,\mfm)^{op}$ denotes the points in $P(\mfi,\mfm)$ with
coordinates reversed, i.e. $P(\mfi,\mfm)^{op} := \{(x_n, \ldots, x_1):  (x_1, \ldots, x_n)
\in P(\mfi,\mfm) \}$.
(The reversal of the ordering on coordinates arises because, locally near $Y_n =
\{[e,e,\ldots,e]\}$ and in Kaveh's coordinates, $Y_i$ is given by the
equations $\{t_{n-i+1} = \cdots = t_n=0\}$, i.e. the \emph{last} coordinates
are $0$. So for example $\nu_1(s)$ is the order of vanishing of $s$
along $\{t_n=0\}$, not $\{t_1=0\}$.)

\begin{theorem}\label{theorem:main} 
Let $\mfi=(i_1, \ldots, i_n) \in \{1,2,\ldots,r\}^n$ be a word and
$\mfm=(m_1, \ldots, m_n) \in \Z^n_{\geq 0}$ be a multiplicity list. Let $Z_\mfi$ and
$L_{\mfi,\mfm}$ denote the associated Bott-Samelson variety and line
bundle respectively. Suppose that $\mfi$
corresponds to a reduced word decomposition
and that 
$(\mfi, \mfm)$ satisfies condition \textbf{(P)}.  Consider the valuation $\nu_{Y_\bullet}$ defined
above and let $S(R(L_{\mfi,\mfm}))$ denote the corresponding value
semigroup. 
Then
\begin{enumerate}
\item the degree-$1$ piece $S_1 := S(R(L_{\mfi,\mfm})) \cap \{1\} \times \Z^n$ of
  $S(R(L_{\mfi,\mfm}))$ is equal to $P(\mfi, \mfm)^{op} \cap \Z^n$
  (where we identify $\{1\} \times \Z^n$ with $\Z^n$ by projection to
  the second factor),
\item $S(R(L_{\mfi,\mfm}))$ is generated by $S_1$, so in particular it is finitely
  generated, and 
\item the Newton-Okounkov body $\Delta = \Delta(R(L_{\mfi,\mfm}))$ of
  $Z_\mfi$ and $L_{\mfi,\mfm}$ with respect to $\nu_{Y_\bullet}$ is
  equal to the polytope $P(\mfi, \mfm)^{op}$. 
\end{enumerate}
\end{theorem}

Before diving into the proof of Theorem~\ref{theorem:main}, we explain
the basic structure of our argument. 
Our first step is 
Proposition~\ref{proposition:subset}, where we show that the image of $\nu_{Y_\bullet}$ is
always a subset of the polytope $P(\mfi, \mfm)^{op}$. This is the most
important step in our argument; here we need that $\mfi$ is reduced. 
Then, under the additional assumption that $(\mfi,\mfm)$ satifies
condition \textbf{(P)}, the results of Section~\ref{sec:bijection}
allows us to quickly conclude 
that $\nu_{Y_{\bullet}}$ gives a surjection from $S_1$ to $P(\mfi,
\mfm)^{op} \cap \Z^n$, from which the theorem follows.

We need some preliminaries. 
For each $j$ with
$1 \leq j \leq n$ let $C_j$ denote the curve in $Z_\mfi$ given by
setting all but the $j$-th coordinate in $[(p_1, \ldots, p_n)] \in
Z_\mfi$ equal to $e$. Note that the curves are isomorphic to $\P^1$. 
The lemma below is from 
\cite[Section 3.7]{GrossbergKarshon1994}.

\begin{lemma}\label{lemma:degree}
Let $\lambda_1, \ldots, \lambda_n$ be a sequence of weights. The degree of the restriction of the line bundle
$L_{\mfi}(\la_1,\dots,\la_n)$ on $Z_\mfi$ to the curve $C_n$ is equal to $\lee \la_n,\be_n^{\vee}\ree$. 
\end{lemma}

In what follows we also need the following codimension-$1$ subvarieties
(divisors) on $Z_\mfi$. For $1 \leq j \leq n$ let 
$Z_{\mfi(j)}$ denote the subvariety of $Z_\mfi$ obtained by requiring the 
$j$-th coordinate of $[(p_1, \ldots, p_n)] \in Z_\mfi$ to be equal to
$e$. Notice that $Z_{\mfi(n)}$ is the same as our $Y_1$
above, and is also naturally isomorphic to the smaller Bott-Samelson
variety $Z_{(i_1, \ldots, i_{n-1})}$ associated to the word obtained
by deleting the last entry in $\mfi$. Also note that 
since $Z_{\mfi(n)}$ is an irreducible subvariety of
codimension $1$, it determines a line bundle $\OO(Z_{\mfi(n)}$. 
We will need the following lemma,
which computes the restriction of certain line bundles on $Z_\mfi$ to
$Z_{\mfi(n)}$. 

\begin{lemma}\label{lemma:restriction of line bundles}
Let $\lambda_1, \ldots, \lambda_n$ be a sequence of weights. Then the restriction to $Z_{\mfi(n)}$ of the line bundle $L_{\mfi}(\la_1,
\ldots, \la_n)$ is isomorphic to $L_{\mfi(n)}(\la_1, \ldots,
\la_{n-2}, \la_{n-1}+\la_n)$ on $Z_{(i_1, \ldots, i_{n-1})}$. Moreover, the restriction 
of $\OO(Z_{\mfi(n)})$ to $Z_{\mfi(n)}$ is isomorphic to
$L_{(i_1,\dots,i_{n-1})}(0,\dots,0,\be_n)$ on $Z_{(i_1,\ldots,i_{n-1})}$.
\end{lemma}

\begin{proof}
Consider the map $\varphi: L_{\mfi(n)}(\la_1, \ldots, \la_{n-1}+\la_n)
\to L_{\mfi}(\la_1, \ldots, \la_n) \vert_{Z_{\mfi(n)}}$ given by
$[(p_1, \ldots, p_{n-1}, k)] \mapsto [(p_1, \ldots, p_{n-1}, e,
k)]$. Then $\varphi$ gives the required isomorphism. Indeed, $\varphi$ is 
well-defined as can be seen by the computation
\begin{equation*}
\begin{split}
[(p_1 b_1, b_1^{-1}p_2 b_2, \ldots, b_{n-2}^{-1}p_{n-1}b_{n-1},
e=b_{n-1}^{-1}b_{n-1}), k] & = 
[(p_1, p_2, \ldots, p_{n-1}, e, e^{-\la_1}(b_1) \cdots
e^{-\la_{n-1}}(b_{n-1}) e^{-\la_n}(b_{n-1}) k)] \\
 & = [(p_1, p_2, \ldots, p_{n-1}, e, e^{-\la_1}(b_1) \cdots
e^{-(\la_{n-1}+\la_n)}(b_{n-1})k)]
\end{split} 
\end{equation*}
in $L_{\mfi}(\la_1, \ldots, \la_n)$. It can be checked similarly that $\varphi$ is injective, and surjectivity is immediate from its definition. 

For the second claim, recall that 
the restriction $\OO(D)\vert_D$ is the normal bundle to $D$ (see
e.g. \cite[Exercise 21.2H]{VakilFOAG}). Applying this to
$Z_{\mfi(n)}$, it suffices to show that the normal bundle to
$Z_{\mfi(n)}$ in $Z_{\mfi}$ is isomorphic to
$L_{(i_1,\dots,i_{n-1})}(0,\dots,0,\be_n)$. Now note $Z_{\mfi}$ is a
$P_{\beta_n}/B$-bundle over $Z_{\mfi(n)} \cong Z_{(i_1, \ldots,
  i_{n-1})}$, and since $Z_{\mfi(n)}$ is defined by setting the last coordinate equal to $e$, 
the normal bundle in question can be identified with $Z_{(i_1, \ldots, i_{n-1})} \times_B T_{eB}(P_{\be_n}/B)$. 
The weight of the action of $B$ on the tangent space $T_{eB}(P_{\be_n}/B)$ at the
identity coset $eB$ of $P_{\be_n}/B$ is $-\be_n$. Thus the normal bundle
is precisely $L_{(i_1,\ldots,i_{n-1})}(0,\dots,0,\be_n)$ as desired. 

\end{proof}

The important step towards the proof of the main
result is the following, which states that the image of the valuation
is contained inside the polytope $P(\mfi, \mfm)^{op}$.

\begin{proposition}\label{proposition:subset}
Let $\mfi=(i_1, \ldots, i_n) \in \{1,2,\ldots,r\}^n$ be a word and $\mfm=(m_1, \ldots, m_n) \in \Z^n_{\geq 0}$
a multiplicity list. 
Let $Z_\mfi$ and $L_{\mfi, \mfm}$ be the Bott-Samelson
variety and line bundle specified by $\mfi, \mfm$ and let
$\nu_{Y_\bullet}$ denote the geometric valuation specified by the flag
$Y_\bullet$ given above. 
Assume that $\mfi$ corresponds to a reduced word
decomposition. Then 
\[
\nu_{Y_\bullet}(H^0(Z_\mfi, L_{\mfi,\mfm}) \setminus \{0\}) \subseteq
P(\mfi,\mfm)^{op} \cap \Z^n.
\]
\end{proposition}

\begin{proof} 
Let $0\ne s\in  H^0(Z_\mfi, L_{\mfi,\mfm})$ with
$\nu_{Y_\bullet}(s)=(x_n,x_{n-1},\dots, x_1)$. We wish to show that
$(x_1, \ldots, x_n) \in P(\mfi,\mfm)$, for which It is enough to
show that $x_n\le m_n$ and $x_k\le A_k(x_{k+1},\dots, x_n)$ for $1
\leq k \leq n-1$. 

We first prove that $x_n\le m_n$. Since 
$m_i \geq 0$ for all $i$, by \cite[Corollary 3.3]{LauritzenThomsen2004} the bundle $L_{\mfi,\mfm}$ is globally generated and hence effective. Moreover, $\mfi$ is reduced by assumption, so we 
can conclude from \cite[Proposition
3.5]{LauritzenThomsen2004} 
that 
\[L_{\mfi,\mfm} \cong \mc{O}\left(\sum_{k=1}^n a_k Z_{\mfi(k)} \right) \] for some
integers 
$a_k\ge 0$, $1 \leq k \leq n$. 
Also  since $x_n=\nu_1(s)=\mr{ord}_{Z_\mfi(n)}(s)$ is the order of
vanishing of $s$ along $Y_1 = Z_{\mfi(n)}$, we know $\mr{div}(s)=x_n
Z_{\mfi(n)} +E$ for some effective divisor $E$. Since $\mr{div}(s)$ is
linearly equivalent to $\sum_{k=1}^n a_k Z_{\mfi(k)}$ we may conclude 
\begin{equation}\label{eq:linear equivalence}
E\sim -x_nZ_{\mfi(n)} +\sum_{k=1}^n a_k Z_{\mfi(k)}
\end{equation}
where $\sim$ denotes linear
equivalence. Considering now the corresponding Chow classes, we may
compare the (intersection) product of both sides of~\eqref{eq:linear
  equivalence} with the class $[C_n] \in A^*(Z_{\mfi})$. The Chow ring
$A^*(Z_{\mfi})$ and the classes $[Z_{\mfi(k)}]$ have been extensively
studied and it is known (cf. \cite{Demazure1974,
  LauritzenThomsen2004}, see also \cite[Proposition 2.11]{Perrin2007})
that $[C_n] \cdot [Z_{\mfi(j)}] = \delta_{jn}$. Thus we obtain
that the product 
$\textup{(RHS of~\eqref{eq:linear equivalence})} \cdot [C_n] = -x_n +
  a_n$, whereas the product $\textup{(LHS of~\eqref{eq:linear equivalence})}
\cdot [C_n] = b_n \geq 0$ since $E$ is effective. 
Hence $x_n \leq a_n$. Furthermore, from \cite[Proposition 2.11]{Perrin2007} and
from basic properties of intersection products, we may
also conclude that 
$a_n$ is the degree of the restriction $L_{\mfi,\mfm} \vert_{C_n}$ of
the line bundle $L_{\mfi,\mfm}$ to the curve $C_n$ (which is
isomorphic to $\P^1$, so $A_0(C_n)\cong \Z$). By
Lemma~\ref{lemma:degree} above,
this degree is precisely equal to $\lee
m_n\varpi_n,\be_n^{\vee}\ree=m_n$. Thus \(x_n \leq m_n\) as was to be
shown.

Next, we consider $x_{n-1}=\nu_2(s)=\mr{ord}_{Y_2}(s_1)$, where $0\ne
s_1\in H^0(Y_1=Z_{\mfi(n)}, L_{\mfi,\mfm}\otimes \OO(-x_nZ_{\mfi(n)})
\vert_{Y_1=Z_{\mfi(n)}})$ and $s_1$ is constructed from $s$ in the
  fashion described above. Note that $Z_{\mfi(n)} \cong Z_{(i_1,
    \ldots, i_{n-1})}$. Thus, repeating the same argument as given
  above, we may deduce that 
  $x_{n-1}$ is at most the degree of the restriction of the line bundle $L_{\mfi,\mfm}\otimes
  \OO(-x_nZ_{\mfi(n)}) \vert_{Y_1 = Z_{\mfi(n)}}$ to the curve
  $C_{n-1}$. 

From Lemma~\ref{lemma:restriction of line bundles} we know that the restriction of $L_{\mfi,\mfm}$ to
$Z_{\mfi(n)} \cong Z_{(i_1, \ldots, i_{n-1})}$ is isomorphic to the
line bundle 
$L_{(i_1, \ldots, i_{n-1})}(m_1\varpi_{\beta_1},\dots,m_{n-2}\varpi_{\beta_{n-2}},m_{n-1}\varpi_{\beta_{n-1}}+m_n\varpi_{\beta_n})$
in the notation of~\eqref{eq:definition line bundle}, and also from
Lemma~\ref{lemma:restriction of line bundles} we know 
\(\OO(Z_n)|_{Z_n}\cong L_{(i_1, \ldots, i_{n-1})} (0,\dots,0,\be_n).\) Thus we have 
\begin{equation}\label{eq:restricted line bundle} 
L_{\mfi,\mfm}\otimes
  \OO(-x_nZ_{\mfi(n)}) \vert_{Y_1 = Z_{\mfi(n)}} \cong L_{(i_1,
    \ldots, i_{n-1})}(m_1 \varpi_1, \ldots, m_{n-2}
  \varpi_{\beta_{n-2}}, m_{n-1}\varpi_{\beta_{n-1}}+m_n\varpi_{\beta_n}
    - x_n \beta_n).
\end{equation}
Since $s_1$ is a non-zero global section, the line bundle
in~\eqref{eq:restricted line bundle} above is effective. Thus by
again applying 
\cite[Proposition
3.5]{LauritzenThomsen2004} we can write it as $\OO(\sum_, a'_k Z_k)$ where
$a'_k \geq 0$. By proceeding with the same argument as before, since
the degree of~\eqref{eq:restricted line bundle} along $C_{n-1}$ is
precisely 
\[\lee
m_{n-1}\varpi_{n-1}+m_n\varpi_n-x_n\be_n,\be_{n-1}^{\vee}\ree=A_{n-1}(x_n)\] 
we may conclude $x_{n-1}\le A_{n-1}(x_n)$. Continuing similarly,
we obtain $(x_1, \ldots,x_n) \in P(\mfi,\mfm)$ as
desired. 
\end{proof}

\begin{remark}\label{remark:scalar multiple}
Note that since a scalar multiple $r\mfm$ is also a multiplicity list for any
positive integer $r$, it immediately follows from the above
proposition that 
\[
\nu_{Y_\bullet}(H^0(Z_\mfi, L_{\mfi,\mfm}^{\otimes r}) \setminus \{0\})
\subseteq P(\mfi, r\mfm)^{op} \cap \Z^n
\]
for any $r\in \N$. 
\end{remark}

To complete the argument we need to recall the following fact from \cite{HaradaYang2014}. 

\begin{proposition}\label{prop:P is lattice}
If $(\mfi,\mfm)$ satisfies condition
\textbf{(P)}, then $P(\mfi,\mfm)$ is a \emph{lattice} polytope. 
\end{proposition}

We are finally ready to prove the main result.

\begin{proof}[Proof of Theorem~\ref{theorem:main}]
We begin with the first claim of the theorem. It is 
  elementary that if a
  valuation $\nu: V \to \Gamma$ (for $V$ a finite-dimensional complex
  vector space and $\Gamma$ a totally ordered group) has
  one-dimensional leaves, then the cardinality $\lvert \nu(V \setminus
  \{0\}) \rvert$ of the image of $\nu$ is equal to $\dim_\C(V)$ 
 \cite[Proposition 2.6]{KavehKhovanskii2012}. Since our valuation
  $\nu_{Y_\bullet}$ has one-dimensional leaves on $R_1$, we conclude
  $\lvert \nu_{Y_\bullet}(R_1 \setminus \{0\}) \rvert =
  \dim_\C(R_1)=\dim_\C(H^0(Z_\mfi, L_{\mfi,\mfm}))$. On the other
  hand, we know from Proposition~\ref{proposition:subset} that the
  image of $\nu_{Y_\bullet}$ on $R_1=H^0(Z_\mfi,L_{\mfi,\mfm})$ must lie
  in $P(\mfi,\mfm)^{op} \cap \Z^n$. 
. Proposition~\ref{proposition:bijection} implies 
  $\lvert P(\mfi,\mfm)^{op}\cap\Z^n \rvert = \lvert
  P(\mfi,\mfm)\cap\Z^n\rvert
  =\mr{dim}_\C(H^0(Z_{\mfi},L(\mfi,\mfm)))$, so we conclude that
  $S_1 := S(R) \cap \{1\} \times \Z^n$ (which by definition is the
  image of $\nu_{Y_\bullet}: R_1 \setminus \{0\} \to P(\mfi,\mfm)^{op}
  \cap \Z^n$) is precisely $P(\mfi, \mfm)^{op} \cap \Z^n$. Here we
  identify $\{1\} \times \Z^n$ with $\Z^n$ by projection to the second
  factor. This proves the first statement of the theorem. 

By
  Remark~\ref{remark:scalar multiple} we also conclude that $S_k$ is equal to $P(\mfi, r
  \mfm)^{op} \cap \Z^n$. From the definition of the polytopes
  $P(\mfi,\mfm)$ it follows that $P(\mfi, r \mfm) = r \cdot
  P(\mfi,\mfm)$. This justifies the second statement of the
  theorem. Finally, the last statement of the theorem now follows
  directly from
Definition~\ref{definition:NO} and Proposition~\ref{prop:P is lattice}.

\end{proof}

\section{Examples}\label{sec:examples}

In this section, we give several concrete examples in order to
illustrate the results in the manuscript. 

Let $G=SL(3,\C)$ with Borel subgroup $B$ the
upper-triangular matrices and $T$ the diagonal subgroup. The rank $r$
is $2$ in this case and we let $\{\alpha_1, \alpha_2\}$ be the usual
positive simple roots corresponding to the simple transpositions $s_1
= (12)$ and $s_2=(23)$ in the Weyl group $W=S_3$. 

For all of the examples below, we consider the Bott-Samelson variety
$Z_\mfi$ where $\mfi=(1,2,1)$ corresponds to the reduced word decomposition
$s_1 s_2 s_1$ of the longest element $w_0$ in $W=S_3$. 

\begin{example} 
Let $\mfm=(1,1,1)$. Then it can be easily checked that $(\mfi=(1,2,1),
\mfm=(1,1,1))$ satisfies
condition \textbf{(P)}. 
The figure below illustrates the polytope $P(\mfi,\mfm)$ which is (up
to a re-ordering of coordinates) the Newton-Okounkov body of
$Z_{(1,2,1)}$ with line bundle $L_{(1,2,1),(1,1,1)}$ with respect to
our valuation $\nu_{Y_\bullet}$. For visualization purposes, the
vertices of the polytope are indicated by black dots, while the other
lattice points are indicated by white dots. 

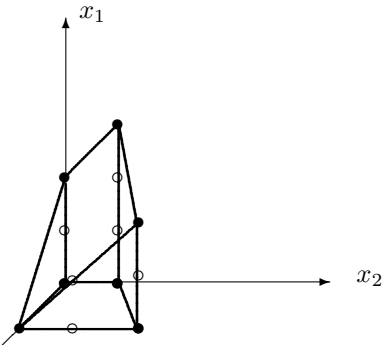
\begin{picture}
  (100,100)(-70,20)
 \put(0,0){\vector(1,0){100}}\put
(0,0){\vector(0,1){100}}\put
(0,0){\vector(-1,-1){30}}

\linethickness{0.25mm}\put(-3,-3){$\bullet$}\put(-3,37){$\bullet$}\put(17,-3){$\bullet$}\put(17,57){$\bullet$}
\put(-20,-20){$\bullet$}\put(25,-20){$\bullet$}\put(25,20){$\bullet$}
\curve(-18,-18,27,-18)\curve(27,-18,27,22)\curve(27,22,20,60)
\curve(20,60,20,0)\curve(20,60,0,40)\curve(0,40,-18,-18)
\curve(-18,-18,27,22)\curve(20,0,27,-18)\curve(0,0,0,37)\curve(-2,0,18,0)
\curve(0,0,-18,-18)

\put(5, 100){$x_1$}\put(110,0){$x_2$}
\put(-35,-40){$x_3$}
\put(-3,-3){$\circ$}\put(17,-3){$\circ$}\put(-3,17){$\circ$}\put(17,17){$\circ$}\put(17,37){$\circ$}
\put(0,-20){$\circ$}\put(25,-20){$\circ$}\put(25,0){$\circ$} \put(0,-2){$\circ$}

\color{blue}\put(120,40)

\end{picture}\\\\\\

\end{example} 

\begin{example} 
Let $\mfm=(2,1,1)$. Again it can be checked easily that $(\mfi,\mfm)$
satifies condition \textbf{(P)}. The polytope $P(\mfi, \mfm)$,
i.e. the Newton-Okounkov body of $Z_\mfi$
and $L_{\mfi,\mfm}$ (again up to reordering of coordinates), is
illustrated below. 

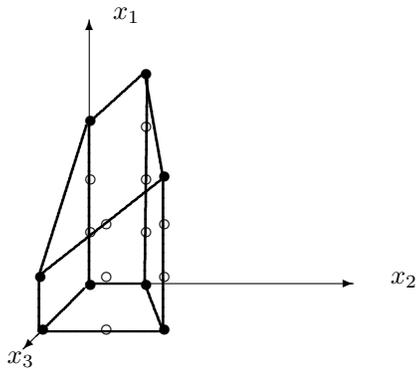
\begin{picture}
   (100,100)(-70,20)
 \put(-4,0){\vector(1,0){100}}\put
(-4,0){\vector(0,1){100}}\put
(-4,0){\vector(-1,-1){25}}

\linethickness{0.25mm}\put(-6,-3){$\bullet$}\put(-6,17){$\circ$}\put(-6,37){$\circ$}
\put(-6,59){$\bullet$}\put(-25,0){$\bullet$}\put(-24,-20){$\bullet$}\put(22,-20){$\bullet$}
\put(22,38){$\bullet$}
\curve(-5,60,-23,3)\curve(-23,3,-23,-18)\curve(-23,-18,24,-18)\curve(24,-18,24,40)
\curve(24,40,17,80)\curve(17,80,-4,61)
\put(15,77){$\bullet$}\curve(24,40,-23,3)\curve(-4,-2,-4,60)\curve(-6,-2,-22,-18)
\curve(-5,0,16,0)\curve(24,-18,17,0)\curve(18,78,17,0)

\put(5, 100){$x_1$}\put(110,0){$x_2$}
\put(-35,-30){$x_3$}
\put(0,0){$\circ$}\put(15,-3){$\bullet$}\put(0,20){$\circ$}\put(15,17){$\circ$}\put(15,37){$\circ$}\put(15,57){$\circ$}
\put(0,-20){$\circ$}\put(22,0){$\circ$} \put(22,20){$\circ$}

\color{red}\put(120,40)

\end{picture}\\\\\\\\

\end{example}

As a final example we consider a
choice of multiplicity list for which the pair $(\mfi,\mfm)$ does
\emph{not} satisfy condition \textbf{(P)}; it can be seen below that
the corresponding $P(\mfi,\mfm)$ is not a lattice polytope. 

\begin{example}\label{example:nonexample}
Let $\mfm = (0,1,1)$.  Then one can check easily that $(\mfi,\mfm)$
does not satisfy condition \textbf{(P)}. The polytope $P(\mfi,\mfm)$
is illustrated below. The vertex which is not a lattice point is
indicated in red. This example was also mentioned in 
Remark~\ref{remark:P' does not imply P}.

\begin{picture}
  (100,100)(-70,20)
 \put(0,0){\vector(1,0){100}}\put
(0,0){\vector(0,1){100}}\put
(0,0){\vector(-1,-1){30}}
\linethickness{0.25mm}
\put(5, 100){$x_1$}\put(110,0){$x_2$}
\put(-35,-35){$x_3$}
\put(-3,-3){$\bullet$}\put(17,-3){$\bullet$}\put(-3,17){$\bullet$}\put(17,17){$\circ$}\put(17,37){$\bullet$}
{\color{red}\put(-10,-10){$\bullet$}}\put(0,-20){$\bullet$}\put(25,-20){$\bullet$}\put(25,0){$\bullet$}
\curve(0,0,-7,-7)\curve(-7,-7,2,-18)\curve(2,-18,27,-18)\curve(27,-18,27,2)\curve(27,2,2,-18)\curve(2,-18,0,20)
\curve(0,20,-7,-7)\curve(0,20,20,40)\curve(20,40,20,0)\curve(20,40,27,2)\curve(20,0,27,-18)
\put(120,40)

\put(-30,-15){\small{$(0,0,\frac{1}{2})$}}
\end{picture}\\\\\\\\\\

\end{example}

\def\cprime{$'$}

\end{document}